\documentclass[a4paper,10pt]{amsart}

\usepackage[colorlinks=true]{hyperref}

\usepackage{verbatim,latexsym,exscale,amssymb,amsthm,amsmath,amsfonts,mathrsfs,amsrefs}
\usepackage{stmaryrd}
\usepackage[all]{xy}

\newdir{ >}{{}*!/-6pt/@{>}}

\numberwithin{equation}{section}
\newcommand{\eg}{\textrm{e.g.}}
\newcommand{\ie}{\textrm{i.e.}}

\newcommand{\cf}{\textrm{cf.}}

\newcommand{\cD}{\mathcal{D}}
\newcommand{\cC}{\mathcal{C}}
\newcommand{\cM}{\mathcal{M}}
\newcommand{\bbF}{\mathbb{F}}

\newcommand{\Pic}{\mathop{\mathrm{Pic}}}

\newcommand{\Hilb}{\mathsf{Hilb}}
\newcommand{\proj}{\mathsf{proj}}

\newcommand{\Map}{\mathrm{Map}} 
\newcommand{\map}{\mathrm{map}} 

\newcommand{\Mor}{\mathrm{Mor}}
\newcommand{\uni}{\mathrm{uni}}

\newcommand{\id}{\mathrm{id}}

\newcommand{\N}{\mathbb{N}}
\newcommand{\op}{\mathrm{op}}
\newcommand{\pt}{\mathrm{pt}}

\newcommand{\unit}{\mathbf 1}
\newcommand{\obj}{\mathop{\mathrm{ob}}}

\newcommand{\Hom}{\mathrm{Hom}}
\newcommand{\Aut}{\mathrm{Aut}}

\newcommand{\Z}{\mathbb{Z}}
\newcommand{\R}{\mathbb{R}}
\newcommand{\C}{\mathbb{C}}
\newcommand{\F}{\mathbb F}

\newcommand{\sSet}{\mathsf{sSet}}

\newcommand{\Cat}{\mathsf{Cat}}

\newcommand{\Cstar}{\mathrm{C}^*\!}
\newcommand{\Cstarmax}{\mathrm{C}^*_{\max}}

\newcommand{\Cstarcatun}{\mathsf{C}_1^*\!\mathsf{cat}}

\newcommand{\Weq}{\mathrm{Weq}}
\newcommand{\Fib}{\mathrm{Fib}}
\newcommand{\Cof}{\mathrm{Cof}}
\newcommand{\Ho}{\mathrm{Ho}}
\newcommand{\cU}{\mathcal{U}}

\newcommand{\add}{\mathsf{add}}

\newcommand{\Fun}{\mathsf{Fun}}

\newcommand{\too}{\longrightarrow}

\numberwithin{equation}{section}

\newtheorem{thm}[equation]{Theorem}
\newtheorem{lemma}[equation]{Lemma}
\newtheorem{thm-defi}[equation]{Theorem-Definition}
\newtheorem{prop}[equation]{Proposition}
\newtheorem{cor}[equation]{Corollary}
\newtheorem*{thm*}{Theorem}
\newtheorem*{lemma*}{Lemma}
\newtheorem*{cor*}{Corollary}

\theoremstyle{definition}

\newtheorem{defi}[equation]{Definition}
\newtheorem*{conv*}{Conventions}

\newtheorem{example}[equation]{Example}

\newtheorem*{conventions*}{Conventions}
\newtheorem*{acknowledgements*}{Acknowledgements}

\theoremstyle{remark}

\newtheorem{remark}[equation]{Remark}

\newtheorem{notation}[equation]{Notation}

\newtheorem*{remark*}{Remark}

\begin{document}
\title{Morita homotopy theory of $\Cstar$-categories\\
}
\author{Ivo Dell'Ambrogio}
\address{Ivo Dell'Ambrogio, Universit\"at Bielefeld, Fakult\"at f\"ur Mathematik, BIREP Gruppe, Postfach 10\,01\,31, 33501 Bielefeld, Germany}
\email{ambrogio@math.uni-bielefeld.de}
\urladdr{www.math.uni-bielefeld.de/~ambrogio/}

\author{Gon\c calo Tabuada}

\address{Gon{\c c}alo Tabuada, Department of Mathematics, MIT, Cambridge, MA 02139, USA}
\email{tabuada@math.mit.edu}
\urladdr{math.mit.edu/~tabuada}

\thanks{}
\address{}
\email{}

\date{\today}

\begin{abstract}
In this article we establish the foundations of the Morita homotopy theory of $\Cstar$-categories. Concretely, we construct a cofibrantly generated simplicial symmetric monoidal Quillen model structure (denoted by $\cM_\Mor$) on the category $\Cstarcatun$ of small unital $\Cstar$-categories. The weak equivalences are the Morita equivalences and the cofibrations are the $\ast$-functors which are injective on objects. As an application, we obtain an elegant description of Brown-Green-Rieffel's Picard group in the associated homotopy category $\Ho(\cM_\Mor)$. We then prove that $\Ho(\cM_\Mor)$ is semi-additive. By group completing the induced abelian monoid structure at each Hom-set we obtain an additive category $\Ho(\cM_\Mor)^{-1}$ and a composite  functor $\Cstarcatun \to \Ho(\cM_\Mor) \to \Ho(\cM_\Mor)^{-1}$ which is characterized by two simple properties: inversion of Morita equivalences and preservation of all finite products. Finally, we prove that the classical Grothendieck group functor becomes co-represented in $\Ho(\cM_\Mor)^{-1}$ by the tensor unit object.
\end{abstract}


\maketitle
\vskip-\baselineskip
\vskip-\baselineskip
\tableofcontents
\vskip-\baselineskip
\vskip-\baselineskip
\vskip-\baselineskip
\section{Introduction}
The theory of $\Cstar$-categories, first developed by Ghez, Lima and Roberts~\cite{glr} in the mid-eighties, has found several useful applications during the last decades. 
Most notably, it has been used by Doplicher and Roberts \cites{doplicher-roberts:new_duality,doplicher-roberts:endo} in the development of a duality theory for compact groups with important applications in algebraic quantum field theory and by Davis and L\"uck \cite{davis-lueck} in order to include the Baum-Connes conjecture into their influential unified treatment of the $K$-theoretic isomorphism conjectures.
Several authors 
 -- see 
 \cites{mitchener:symm, mitchener:Cstar_cats, mitchner:KKth, mitchener:gpd} 
\cite{joachim:KC*} 
 \cites{kandelaki:KK_K, kandelaki:multiplier, kandelaki:karoubi_villamayor, kandelaki:fredholm} 
 \cites{vasselli:bundles, vasselli:bundlesII} 
 \cite{zito:2C*}
 \ldots --
have subsequently picked up these strands of ideas and employed $\Cstar$-categories in various operator-theoretic contexts, often in relation to Kasparov's $KK$-theory.
In most of the above situations $\Cstar$-categories are only to be considered up to {\em Morita equivalence}, the natural extension of the classical notion of Morita(-Rieffel) equivalence between $\Cstar$-algebras. Hence it is of key importance to development the foundations of a Morita theory of $\Cstar$-categories.

A $*$-functor $F:A\to B$ between $\Cstar$-categories is called a \emph{Morita equivalence} if it induces an equivalence $F^\natural_\oplus: A^\natural_\oplus \to B^\natural_\oplus$ on the completions of $A$ and $B$ under finite direct sums and retracts. This operation $A\mapsto A^\natural_\oplus$, named \emph{saturation}, can easily be performed without leaving the world of $\Cstar$-categories; see~\S\ref{subsec:saturated}.

The first named author has initiated in~\cite{ivo:unitary} the study of $\Cstar$-categories via homo\-to\-py-theo\-ret\-ic methods,  in particular by constructing the unitary model structure, where the weak equivalences are the unitary equivalences of $\Cstar$-categories. 
In the present article we take a leap forward in the same direction by establishing the foundations of the Morita homotopy theory of $\Cstar$-categories. Our first main result, which summarizes Theorem~\ref{thm:morita_model}, Propositions \ref{prop:monoidal}, \ref{prop:simplicial} and \ref{prop:bousfield}, and Corollary~\ref{cor:fibrant_replacement}, is the following:

\begin{thm}
\label{thm:main1}
The category $\Cstarcatun$ of (small unital) $\Cstar$-categories and (identity preserving) $*$-functors admits a Quillen model structure whose weak equivalences are the Morita equivalences and whose cofibrations are the $*$-functors which are injective on objects. Moreover, this model structure is cofibrantly generated, symmetric monoidal, simplicial, and is endowed with a functorial fibrant replacement given by the saturation functor. Furthermore, it is a left Bousfield localization of the unitary model structure.
\end{thm}

We have named this Quillen model the {\em Morita model category} of $\Cstar$-categories (and denoted it by $\cM_\Mor$) since two unital $\Cstar$-algebras become isomorphic in the associated homotopy category $\Ho(\cM_\Mor)$ if and only if they are Morita equivalent ($=$~Morita-Rieffel equivalent) in the usual sense; see Proposition~\ref{prop:Morita_agreement}. The Morita homotopy category $\Ho(\cM_\Mor)$ becomes then the natural setting where to formalize and study all ``up to Morita equivalence'' phenomena. As an example we obtain an elegant conceptual description of the Picard group (see \S\ref{sec:Pic}).
\begin{prop}\label{prop:Pic}
For every unital $\Cstar$-algebra~$A$ there is a canonical isomorphism
\begin{equation}\label{eq:Pic}
\Aut_{\Ho(\cM_\Mor)}(A) \simeq \Pic(A)
\end{equation}
between its automorphism group in the Morita homotopy category and its Picard group $\Pic(A)$, as originally defined  by Brown-Green-Rieffel in \cite{brown-green-rieffel} using imprimitivity bimodules.
\end{prop}
As a consequence, the left-hand-side of \eqref{eq:Pic} furnishes us with a simple Morita invariant definition of the Picard group of {\em any} $\Cstar$-category. Our second main result, which summarizes Theorems~\ref{thm:semi-additive} and~\ref{thm:map_sum}, Proposition~\ref{prop:morphisms_Morita}, and Corollary~\ref{cor:new}, is the following:
\begin{thm}\label{thm:main2}
The homotopy category $\Ho(\cM_\Mor)$ is {\em semi-additive}, \ie\ it has a zero object, finite products, finite coproducts, and the canonical map from the coproduct to the product is an isomorphism. Its Hom-sets admit the following description 
\[
 \Hom_{\Ho(\cM_\Mor)}(A,B) \simeq  \mathrm{ob}(\Cstar(A,B^\natural_\oplus))/_{\!\simeq} \;,
 \]
where $\Cstar(A,B^\natural_\oplus)$ denotes the $\Cstar$-category of $\ast$-functors from $A$ to the saturation of~$B$ and the equivalence relation $\simeq$ on objects is unitary isomorphism. Moreover, the canonical abelian monoid structure thereby obtained on each Hom-set $\Hom_{\Ho(\cM_\Mor)}(A,B)$ is induced by the direct sum operation on~$\smash{B^\natural_\oplus}$.
\end{thm}
Intuitively speaking, Theorem~\ref{thm:main2} shows us that by forcing Morita invariance we obtain a local abelian monoid structure. By group completing each Hom monoid we obtain then an additive category $\Ho(\cM_\Mor)^{-1}$ and hence a composed functor 
\begin{equation*}\label{eq:comp}
\cU\colon \Cstarcatun \too \Ho(\cM_\Mor) \too \Ho(\cM_\Mor)^{-1}\,;
\end{equation*} 
consult \S\ref{sec:K0} for details. 
Our third main result (see Theorem~\ref{thm:characterization}) is the following:

\begin{thm}\label{thm:main3}
The canonical functor $\cU$ takes values in an additive category, inverts Morita equivalences, preserves all finite products, and is universal among all functors having these three properties.
\end{thm}

Our last main result, collecting Theorem \ref{thm:co-representability} and Proposition \ref{prop:ring_comm}, provides a precise link between our theory and the $K$-theory of $\Cstar$-algebras.
Note that, since $\cM_\Mor$ is symmetric monoidal, its tensor structure descends to $\Ho(\cM_\Mor)$ and then extends easily to the group completion $\Ho(\cM_\Mor)^{-1}$.

\begin{thm}
\label{thm:main4}
For every unital $\Cstar$-algebra~$A$ there is a canonical isomorphism 
\begin{equation} \label{eq:Kthformula}
\Hom_{\Ho(\cM_\Mor)^{-1}}(\bbF,A) \simeq K_0(A)
\end{equation}
of abelian groups, where the right-hand-side denotes the classical Grothendieck group of~$A$.
When $A=C(X)$ is moreover commutative, the usual ring structure on $K_0(A)$ (induced by the tensor product of vector bundles) coincides with the to one obtained on the left-hand-side by considering $A$ as a ring object in the symmetric monoidal category $\Ho(\cM_\Mor)^{-1}$.
\end{thm}
As a consequence the left-hand-side of \eqref{eq:Kthformula} provides us with an elegant Morita invariant definition of the Grothendieck group of {\em any} $\Cstar$-category. Note that by Theorem~\ref{thm:main3} this definition is completely characterized by a simple universal property. In Remark~\ref{remark:Ktheory} we compare our approach with those of other authors.

\subsection*{Conventions}
We use the symbol~$\F$ to denote the base field, which is fixed and is either $\R$ or~$\C$. Except when stated otherwise, all $\Cstar$-categories are \emph{small} (they have a set -- as opposed to a class -- of objects) and \emph{unital} (they have an identity arrow $1_x$ for each object~$x$). Similarly, all $*$-functors are unital (they preserve the identity maps~$1_x$).
We will generally follow the notations from~\cite{ivo:unitary}.



\section{Direct sums and idempotents}
In this section we consider, in the context of $\Cstar$-categories, the additive hull and the idempotent completion constructions. Both will play a central role in the sequel. 

\subsection{$\F$-categories, $*$-categories, Banach categories, and $\Cstar$-categories}
\label{subsec:basics}
For the reader's convenience, we start by recalling some standard definitions and facts; consult \cite{ivo:unitary}*{\S1} for more details and examples.
Recall that an \emph{$\F$-category} is a category enriched over $\F$-vector spaces (see \cite{kelly:enriched_book}); concretely, each Hom-set carries an $\F$-vector space structure for which composition is $\F$-bilinear.
A \emph{$*$-category}~$A$ is an $\F$-category which comes equipped with an \emph{involution} $a\mapsto a^*$ on arrows. More precisely, the involution is a conjugate-linear contravariant endofunctor on~$A$ which is the identity on objects and which is its own inverse. The arrow $a^*$ is called the \emph{adjoint} of~$a$. 
A \emph{Banach category} is an $\F$-category where moreover each Hom-space is a Banach space in such a way that $\|b\circ a\|\leq \|b\|\|a\|$ for all composable arrows and $\|1_x\|=1$ for all identity arrows.
A \emph{$\Cstar$-category} $A$ is simultaneously a Banach category and a $*$-category, where moreover the norm is a \emph{$\Cstar$-norm}, \ie\ for every arrow $a\in A(x,x')$ we require that: 
\begin{itemize}
\item[(i)] the $\Cstar$-equality $\|a^*a\|=\|a\|^2$ holds;
\item[(ii)] the arrow $a^*a$ is a positive element of the endomorphism $\Cstar$-algebra $A(x,x)$, \ie\ its operator-theoretic spectrum is contained in $[0,\infty[\,\subset \F$.
\end{itemize}
A \emph{$*$-functor} $F\colon A\to A'$ is a functor preserving the $\F$-linear structure and the involution. If $A,A'$ are $\Cstar$-categories, $F$ will automatically be norm-reducing on each Hom-space, and if $F$ is moreover faithful (\ie\ injective on arrows) it will automatically be isometric, \ie\ norm preserving: $\|F(a)\|=\|a\|$ (the converse being obvious). 
\begin{notation}
The category of all (small) $\Cstar$-categories and all (identity preserving) $*$-functors will be denoted by $\Cstarcatun$.
\end{notation}
\begin{example} \label{ex:C*cats}
Every (unital) $\Cstar$-algebra~$B$ can be identified with the (small) $\Cstar$-category with precisely one object $\bullet$ and whose endomorphisms algebra is given by $B(\bullet,\bullet):=B$. Then a $*$-functor $B\to B'$ between unital $\Cstar$-algebras is the same as a unital $*$-homomorphism.
The collection of all Hilbert spaces and all bounded linear operators between them (together with the operator norm and the usual adjoint operators) is an example of a (large) $\Cstar$-category $\Hilb$.
\end{example}
The axioms of a $\Cstar$-category are designed so that the following basic standard result holds: every small $\Cstar$-category admits a concrete realization as a sub-$\Cstar$-category of~$\Hilb$. This is essentially the GNS construction; see \cite{glr}*{Prop.\,1.14}. The converse is also clear: every norm-closed $*$-closed subcategory of $\Hilb$ inherits the structure of a $\Cstar$-category.

\begin{notation}
\label{notation:underlying}
Given a $\Cstar$-category $A$ (or a $*$-category, Banach category,\ldots), we denote by $UA$ its \emph{underlying category} that one obtains by simply  forgetting some of its structure. This defines a faithful functor $U\colon \Cstarcatun\to \Cat$
from $\Cstar$-categories to ordinary (small) categories and ordinary functors. 
Similarly, we have a forgetful functor $U_\F\colon\Cstarcatun\to \F\textrm-\Cat$ to (small) $\F$-categories and $\F$-linear functors between them.
\end{notation}

\subsection{Some $*$-categorical notions}
\label{subsec:*-cat}
In the context of $\Cstar$-categories, or more generally $*$-categories, it is natural to require all categorical properties and constructions to be compatible with the involution. The following notions concern objects, morphisms, and more generally diagrams \emph{inside} a given $\Cstar$-category and the terminology is inspired by the example $\Hilb$.

A \emph{unitary} morphism (or \emph{$*$-isomorphism}) is an invertible morphism $u\colon x\to y$ such that $u^{-1}=u^*$.
We say that two objects are \emph{unitarily isomorphic} if there exists a unitary morphism between them.
An  \emph{isometry} (or \emph{$*$-mono}) is an arrow $v\colon r\to x$ such that $v^*v=1_r$.
We call~$r$ a \emph{retract of~$x$} (or \emph{$*$-retract}); whenever we say that $r$ is a retract of~$x$, we will assume that an isometry $v\colon r\to x$ has been specified.
Dually, a \emph{coisometry} (or \emph{$*$-epi}) is a morphism $w\colon x\to r$ such that $ww^*=1_r$; note that $v$ is an isometry if and only if $v^*$ is a coisometry.


A \emph{projection} (or \emph{$*$-idempotent}) is a self-adjoint idempotent morphism $p=p^2=p^*$. 
If $v\colon r\to x$ is an isometry, then $p=vv^*$ is a projection on~$x$. In this case we say that $p$ has \emph{range object}~$r$, or that $p$ is the \emph{range projection} of~$r$. 
If $v\colon r\to x$ and $v'\colon r'\to x$ are two retracts of the same object~$x$, it follows that $r$ and~$r'$ are range objects for the same projection (\ie\ $vv^*=v'v'^*$) if and only if there exists a unitary isomorphism $u\colon r\stackrel{\sim}{\to}r'$ such that $v'u=v$. Therefore, if the range object of a projection exists then it is uniquely determined up to a unique unitary isomorphism.

\begin{remark}
Note that identity maps are both projections and unitary isomorphisms.
Moreover, in a $\Cstar$-category all unitaries, isometries and projections automatically have norm one or zero; indeed, this is well-known for bounded operators between Hilbert spaces, and as we have recalled every $\Cstar$-category is isomorphic to a $\Cstar$-category of such operators. (Generally speaking, this reasoning is a quick way to gain some intuition on $\Cstar$-categories for those familiar with Hilbert spaces.)
\end{remark}

A \emph{direct sum} (or \emph{$*$-biproduct}) of finitely many objects $x_1,\ldots, x_n$ is an object $x_1\oplus\cdots \oplus x_n$ together with isometries $v_i:x_i\to x_1\oplus\cdots \oplus x_n$ $(i=1,\ldots,n)$ such that the following equations hold: 
\begin{eqnarray*}
v_1v_1^*+ \ldots + v_n v_n^*=1_{x_1\oplus \cdots \oplus x_n}
&&  
v_i^* v_j = \delta_{ij} \quad (i,j\in \{1,\ldots,n\})\,,
\end{eqnarray*}
where $\delta_{ij}$ is the evident Kronecker delta: $\delta_{ij}=1_{x_i}$ if $i=j$ or $0$ otherwise. 

\begin{remark}
\label{remark:dir_sums}
By definition, the direct sum  $x_1\oplus \cdots \oplus x_n$ is also a biproduct in the underlying $\F$-category, and thus both a product and a coproduct in the underlying category~$UA$.
Moreover, in analogy with retracts, a direct sum is uniquely determined up to a unique \emph{unitary} isomorphism.
\end{remark}

\begin{remark}
A $*$-functor  between $*$-categories preserves each one of the above notions.
\end{remark}

\begin{defi}\label{def:unitary}
Let $A$ and $B$ be two $*$-categories.
A \emph{unitary equivalence} (or \emph{$*$-equivalence}) between $A$ and~$B$ is a $*$-functor $F\colon A\to B$ for which there exist a $*$-functor $G\colon B\to A$ and natural \emph{unitary} isomorphisms $FG\simeq \id_B$ and $GF\simeq \id_A$. 
Equivalently, $F$ is fully faithful and \emph{unitarily}
essentially surjective, \ie\ for every $y\in \obj (B)$ there exists an $x\in \obj( A)$ and a unitary isomorphism $u\colon Fx\stackrel{\sim}{\to} y$ in~$B$.
\end{defi}
It can be shown that any two objects in a $\Cstar$-category are isomorphic if and only if they are unitarily isomorphic; see \cite{ivo:unitary}*{Prop.~1.6}. In particular, a $*$-functor $F\colon A\to B$ between $\Cstar$-categories is unitarily essentially surjective if and only if it is essentially surjective in the usual sense. Therefore, it is a unitary equivalence if and only if it is an equivalence of the underlying categories (that is, if and only if $UF\colon UA\stackrel{\sim}{\to} UB$ is an equivalence of categories). 
\subsection{Adding direct sums}


Let $A$ be a $\Cstar$-category. We say that $A$ is \emph{additive} if its underlying $\F$-category $U_\F A$ is additive. This amounts to requiring that $U_\F A$  admits biproducts, or equivalently that $UA$ admits a zero object, all finite products and coproducts, and that the canonical maps comparing coproducts with products are isomorphisms (\cf\ \S\ref{sec:semi_add}). 
As it will become apparent in what follows, this is the same as requiring that $A$ admits all finite ($*$-compatible) direct sums.

\begin{defi}[Additive hull $A_\oplus$; see \cite{mitchener:symm}*{Def.~2.12}]
\label{def:add_hull}
The \emph{additive hull} $A_\oplus$ of $A$ is the $\Cstar$-category defined as follows: the objects are the formal words $x_1\cdots x_n$ on the set $\obj(A)$ and the Hom-spaces are the spaces of matrices, written as follows: 
\[
A(x_1\cdots x_n, y_1\cdots y_m) := \bigoplus_{ \substack{j=1,\ldots ,n \\ i=1,\ldots, m}} A(x_j, y_i)  \; \ni \;[a_{ij}]\,.
\]
Composition is the usual matrix multiplication, $[b_{ij}]\circ [a_{ij}] = [\sum_k b_{ik} a_{kj} ]$, and adjoints are given by the conj\-ugate-trans\-pose $[a_{ij}]^*:= [a_{ji}^*]$. There exists a unique $\Cstar$-norm on $A_\oplus$ making the canonical fully faithful $*$-functor
\begin{equation*}
\sigma_A\colon
A \too A_\oplus 
\quad\quad 
x \mapsto x
\quad\quad
a \mapsto a=[a]
\end{equation*}
isometric. Moreover, $A_\oplus$ is complete 
for this norm, \ie\ it is a Banach category and so in fact a $\Cstar$-category. 
(For a quick proof of these facts choose a faithful, and hence isometric, representation $F\colon A\to \Hilb$. Since the $*$-functor~$F$ has an evident extension to~$A_\oplus$, one can now argue with bounded operators).
Given a $*$-functor $F: A\to B$, we define a $*$-functor $F_\oplus : A_\oplus \to B_\oplus$ by setting
\[
F_\oplus(x_1\cdots x_n):= F(x_1) \cdots F(x_n)
\quad \textrm{ and } \quad
F([a_{ij}]):= [F(a_{ij})]
\]
for all objects $x_1\cdots x_n\in \obj(A_\oplus)$ and arrows $[a_{ij}]\in A_\oplus$. We obtain in this way a well-defined \emph{additive hull functor} $(-)_\oplus: \Cstarcatun \to \Cstarcatun$.
\end{defi}

\begin{remark}
Note that the object $x=x_1 \cdots x_n$, together with the evident matrices
\[
v_i=
\left[
\begin{array}{c}
0\\
\vdots \\
1_{x_i} \\
\vdots \\
0
\end{array}
\right] \colon x_i \to x
\quad\quad\quad
v^*_i =
\left[
\begin{array}{ccccc}
0 & \cdots & 1_{x_i} & \cdots & 0
\end{array} \right]
\colon x \to x_i 
\;,
\]
is a canonical choice for the direct sum in $A_\oplus$ of the objects $x_1,\ldots, x_n$. In particular, the empty word provides a zero object~$0$.
Hence $A$ admits all finite direct sums and thus it is additive (\cf\ Remark \ref{remark:dir_sums}).  Note also that $A$ admits all finite direct sums if and only if $\sigma_A\colon A\to A_\oplus$ is a unitary equivalence. For this use the fact that $\sigma_A$ is fully faithful and that direct sums are unique up to a unitary isomorphism.
\end{remark}

\begin{notation}
In the following, whenever we write $x_1\oplus\cdots \oplus x_n $ in some additive hull $A_{\oplus}$, we mean the canonical direct sum $ x_1\cdots x_n$ with the above matrix isometries. Similarly, by~$0$ we will always mean the empty word.
\end{notation}

\begin{remark}\label{remark:oplus_agreement}
If we ignore norms and adjoints, the same precise construction as in Definition~\ref{def:add_hull} provides an additive hull $C_\oplus$ for any $\F$-linear category~$C$. Hence~$C$ is additive (\ie\ admits finite biproducts) if and only if $\sigma_C\colon C\to C_\oplus$ is an equivalence. Note that in the case of a $\Cstar$-category~$A$ we have the equality $U_\F(A_\oplus) = (U_\F A)_\oplus$.
\end{remark}

The additive hull can be characterized by the following 2-universal property.

\begin{lemma}
\label{lemma:UP_dirsum}
Let $B$ be an additive $\Cstar$-category. 
Then the induced $*$-functor 
$$\sigma_A^*=\Cstar(\sigma_A,B)\colon \Cstar(A_\oplus, B)\to \Cstar(A,B)$$
is a unitary equivalence.
\end{lemma}
Here, $\Cstar(-,-)\colon \Cstarcatun^\op\times \Cstarcatun\to \Cstarcatun$ denotes the internal Hom functor, which for two $\Cstar$-categories $C,D$ yields the $\Cstar$-category $\Cstar(C,D)$ of $*$-functors $F\colon C\to D$ and bounded natural transformations first introduced in \cite{glr}*{Prop.\,1.11}; see also~\cite{ivo:unitary}.
\begin{proof}
Every $*$-functor $F\colon A\to B$ extends along $\sigma_A$ by the formula $\widetilde{F}(x_1\cdots x_n):=F(x_1)\oplus\cdots \oplus F(x_n)$, which requires the choice of direct sums in~$B$. Nonetheless, the extension is unique up to unitary isomorphism of $*$-functors.
Every bounded natural transformation $\alpha\colon F\to G$ extends diagonally to a bounded natural transformation $\widetilde \alpha\colon \widetilde F\to \widetilde G$ and the extension is unique since every (bounded) natural transformation $\beta \colon \widetilde F\to \widetilde G$ must be diagonal: if $x_1,\ldots,x_n\in \obj(A)$ and $i\neq j$ then 
\[
\xymatrix{
Fx_j \ar@/^3ex/[rr]^-0 \ar[d]_{\beta_{x_j}} \ar[r]_-{Fs_j} & 
 \widetilde{F}(x_1\cdots x_n) \ar[r]_-{Fs_i^*} \ar[d]|{\beta_{x_1\cdots x_n}} &
   Fx_i \ar[d]^{\beta_{x_i}}  \\
Gx_j \ar@/_3ex/[rr]_0  \ar[r]^-{Gs_j} &
 \widetilde{G}(x_1\cdots x_n) \ar[r]^-{Gs_i^*} & 
  Gx_i
}
\]
commutes by naturality, showing that $\beta_{x_1\cdots x_n}$ has zero off-diagonal components.
This implies that the $*$-functor $\Cstar(\sigma_A,B)$ is a unitary equivalence.
\end{proof}

\subsection{Splitting idempotents}
\label{subsec:idemp}

Recall that an $\F$-category 
is said to be \emph{idempotent complete} 
if every idempotent endomorphism $e=e^2:x\to x$ \emph{splits},  \ie\ if $e$ has a kernel and a cokernel. Equivalently, every idempotent has an image, or (in the presence of direct sums) for every idempotent there exist two objects $x',x''\in \obj(A)$ and an isomorphism $x\simeq x'\oplus x''$ 
identifying the given idempotent with the map $\big[{}^1_0 \, {}^0_0 \big] : x'\oplus x'' \to x'\oplus x''$. Then~$x'$ is the image (range) of~$e$ and $x''$ provides its kernel as well as its cokernel.

\begin{lemma}
\label{lemma:idem_vs_proj}
Let $e=e^2\in A(x,x)$ be an idempotent morphism in a $\Cstar$-category~$A$.
Then there exists an automorphism $a\in A(x,x)$ such that $a^{-1}ea$ is a projection.
\end{lemma}

\begin{proof}
The endomorphism ring $A(x,x)$ is a $\Cstar$-algebra and the result is well-known for $\Cstar$-algebras, consult for example~\cite{blackadar:Kth_op_alg}*{Prop.~4.6.2}. 
\end{proof}

\begin{cor}
\label{cor:idem}
Given a $\Cstar$-category $A$, the following conditions are equivalent:
\begin{enumerate}
\item[(i)] The underlying $\F$-category of $A$ is idempotent complete;

\item[(ii)] Every projection $p=p^*=p^2\in A$ splits;

\item[(iii)] For every projection $p\in A$ there is an isometry $v$ such that $vv^*=p$. 

\end{enumerate}
\end{cor}

\begin{proof}
Equivalence (i)$\Leftrightarrow$(ii) follows from Lemma \ref{lemma:idem_vs_proj} and from the fact that similar (\ie\ isomorphic) idempotents have isomorphic (co)kernels. Equivalence (i)$\Leftrightarrow$(iii) follows from the matrix description of a split idempotent that we have recalled above (computed in $A_\oplus$ if necessary),  and the fact that the isomorphism $x\simeq x'\oplus x''$ can be suitably chosen to be unitary (\cf\ the proof of \cite{ivo:unitary}*{Prop.~1.6}).
\end{proof}

An interesting consequence of Corollary~\ref{cor:idem} is the fact that in order to idempotent complete a $\Cstar$-category it suffices to split its projections. We implement this as follows (\cf\ \cite{karoubi:K}*{\S6} and \cite{kandelaki:KK_K}*{\S2.2}):
\begin{defi}[Idempotent completion $A^\natural$]
Let $A$ be a $\Cstar$-category. Its \emph{idempotent completion} $ A^\natural$ is the $\Cstar$-category defined as follows: 
\begin{align*}
& \obj (A^\natural )
 :=  \{(x,p)\mid x\in \obj( A) \textrm{ and } p=p^*=p^2\in A(x,x) \}  \\
& A^\natural((x,p), (x',p')) 
: =   p' A(x,x') p = \{a \in A(x,x') \mid p'a= a = ap\}\,.
\end{align*}
Composition is induced from $A$, and the identity morphism of an object $(x,p)$ is~$p$.   
There is a canonical fully faithful $*$-functor $\tau_A:A \to A^\natural$, sending an object~$x$ to $(x,1_x)=:x$. Given a $*$-functor $F: A\to B$, we define the $*$-functor $F^\natural: A^\natural \to B^\natural$ by setting
\[
F^\natural(x,p) := (Fx, Fp)
\quad \textrm{ and } \quad
F^\natural a:= Fa 
\]
for all objects $x\in\obj(A)$ and arrows $a\in A$. We obtain then a well-defined \emph{idempotent completion functor} $(-)^\natural \colon \Cstarcatun\to \Cstarcatun$.
\end{defi}
\begin{remark}
By construction $A^\natural$ is always idempotent complete: the range of a projection~$q$ on $(x,p)$ is $(x,q)$ with isometry $v=q\in A^\natural((x,q),(x,p))$, as one verifies immediately. In fact, $A$ is idempotent complete if and only if $\tau_A\colon A\to A^\natural$ is a unitary equivalence.
This follows from the fact that $\tau_A$ is fully faithful and that range objects are unique up to a unitary isomorphism.
\end{remark}
The idempotent completion admits the following 2-universal property.
\begin{lemma}
\label{lemma:UP_proj}
Let $B$ be an idempotent complete $\Cstar$-category \textup(\ie\ it admits range objects for all projections\textup). Then, the induced $*$-functor 
$$\tau_A^*=\Cstar(\tau_A,B)\colon \Cstar(A^\natural, B)\stackrel{\sim}{\to} \Cstar(A,B)\,,$$
is an unitary equivalence.
\end{lemma}
\begin{proof}
The proof is similar to the one of Lemma \ref{lemma:UP_dirsum} and is left as an exercise for the reader; consult \cite{karoubi:K}*{Theorem~I.6.10} if necessary.
\end{proof}

\subsection{Saturated $\Cstar$-categories}
\label{subsec:saturated}
We now combine the ideas of the two previous subsections.

\begin{defi}
\label{defi:saturated}
A $\Cstar$-category $A$ is called \emph{saturated} if it is additive and idempotent complete. 
Equivalently, $A$ is saturated if it admits a zero object and direct sums, and moreover all its projections split (see Corollary~\ref{cor:idem}). 
\end{defi}

\begin{remark}\label{rk:sums}
If a $\Cstar$-category $A$ admits all finite direct sums, then the same holds for~$A^\natural$, with
$(x_1,p_1)\oplus \cdots \oplus (x_n,p_n) = (x_1\oplus\cdots \oplus x_n, p_1\oplus\cdots\oplus p_n)$. Hence, for every $\Cstar$-category $B$, the associated $\Cstar$-category $(B_\oplus)^\natural$ is always saturated.
Note that in general the operations $(-)_\oplus$ and $(-)^\natural$ do \emph{not} commute since $(B^\natural)_\oplus$ doesn't need to be idempotent complete (as the simple example $B=\F$ shows). 
\end{remark}
\begin{notation}
Taking into account Remark~\ref{rk:sums}, we will reserve the symbol $(-)^\natural_\oplus$ for the composition $ ((-)_\oplus)^\natural$.
\end{notation}
\begin{defi}
\label{defi:saturated1}
Given a $\Cstar$-category~$A$, consider the following composite
\begin{equation}\label{eq:canonical}
\iota_A\; \colon \;
\xymatrix{
A \ar[r]^-{\sigma_A} & A_\oplus \ar[r]^-{\tau_{A_\oplus}} & (A_\oplus)^\natural =:A_\oplus^\natural \;.
}
\end{equation}
The $\Cstar$-category $A^\natural_\oplus$ will be called the \emph{saturation of~$A$}. We thus obtain a well-defined \emph{saturation functor} $(-)^\natural_\oplus\colon \Cstarcatun\to \Cstarcatun$, which comes equipped with the natural augmentation $\iota\colon \id\to (-)^\natural_\oplus$.
\end{defi}

\begin{prop}
\label{prop:UP_saturated}
The saturation functor verifies the following properties:
\begin{enumerate}
\item[(i)] \emph{(2-universal property.)}
Given a saturated $\Cstar$-category $B$, precomposition with $\iota_A\colon A\to A^\natural_\oplus$ induces a unitary equivalence $\Cstar(A^\natural_\oplus, B)\stackrel{\sim}{\to} \Cstar(A,B)$. 
\item[(ii)] A $\Cstar$-category $A$ is saturated if and only if $\iota_A$ is a unitary equivalence.
\item[(iii)] A $*$-functor $F: A\to B$ is fully faithful if and only if $F^\natural_\oplus$ is fully faithful.
\item[(iv)] If $F$ is a unitary equivalence, then so is $F^\natural_\oplus$.
\end{enumerate}
\end{prop}

\begin{proof}
In what concerns property (i), simply combine Lemma \ref{lemma:UP_dirsum} and Lemma~\ref{lemma:UP_proj}. The $*$-functor $\iota_A$ is fully faithful and preserves whatever retracts and direct sums exist in~$A$. 
Since direct sums and range objects of projections are unique up to unitary isomorphism, we observe that if~$A$ is saturated then $\iota_A$ is unitarily essentially surjective and therefore a unitary equivalence. The converse is similar and so the proof of property (ii) is achieved. In what concerns property (iii), consider the following commutative diagram
\[
\xymatrix{
A \ar[d]_F \ar[r]^-{\iota_A} & A^\natural_\oplus \ar[d]^{F^\natural_\oplus} \\
B \ar[r]_-{\iota_B} & B^\natural_\oplus\,.
}
\]
Since $\iota_A$ and $\iota_B$ are fully faithful, we observe that if $F^\natural_\oplus$ is fully faithful then so is~$F$. 
On the other hand, given any two objects 
$(x,p)$ and $(x',p')$ of~$A^\natural_\oplus$ with $x=x_1\cdots x_n$ and $x'=x'_1\cdots x'_{n'}$ for some $x_i,x'_i\in\obj( A)$ and with $p=[p_{ij}]$ and $p'=[p'_{ij}]$ for some self-adjoint idempotent matrices, the component 
\[
F^\natural_\oplus= F^\natural_\oplus((x,p),(x',p'))\colon A^\natural_\oplus((x,p),(x',p'))\to B^\natural_\oplus(F(x,p),F(x',p'))
\] 
of the functor $F^\natural_\oplus$ is by definition equal to the map
\begin{eqnarray*}
q \Big( \bigoplus_{i,j} A(x_i,x_j) \Big)p\too Fq \Big(\bigoplus_{i,j} B(Fx_i,x_j) \Big)Fp && [a_{ij}] \mapsto [Fa_{ij}]\,.
\end{eqnarray*}
Clearly the latter map is bijective if each component $F=F(x_i,x_j)\colon A(x_i,x_j)\to B(x_i,x_j)$ of~$F$ is bijective, and this proves property~(iii). In what concerns property~(iv), it suffices by (iii) to prove that $F^\natural_\oplus$ is unitarily essentially surjective. This is straightforward and left to the reader. 
\end{proof}

\begin{remark}[Cf.\ \cite{kandelaki:KK_K}*{Lemma 10}]
\label{remark:idemp_agreement}
Let $A$ be a $\Cstar$-category and $U_\F A$ its underlying $\F$-category. Then we may want to saturate $U_\F A$ in the usual algebraic sense, namely, we can construct the idempotent complete $\F$-linear hull $\iota_{U_\F A}\colon U_\F A\to ((U_\F A)_\oplus)^\flat$, where $(-)^\flat$ denotes the original version of the idempotent completion, or ``pseudo-abelian'' hull \cite{karoubi:K}*{\S I.6}, where objects are all pairs $(x,e)$ with $x\in \obj UA$ and $e=e^2\in UA(x,x)$ any idempotent map. Since projections are a particular kind of idempotents (and by Remark \ref{remark:oplus_agreement}), we obtain an evident comparison functor 
\[
\xymatrix@R=3ex@C=3ex{
& U_\F A \ar[dl]_{U_\F \iota_A} \ar[dr]^{\iota_{U_\F A}} & \\
U_\F (A_{\oplus}^\natural) \ar@{-->}[rr]^-\sim && (U_\F A_\oplus)^\flat\,.
}
\]
This functor is an equivalence of categories since it is fully faithful by construction and essentially surjective by Lemma~\ref{lemma:idem_vs_proj}.
\end{remark}

\section{Key $\ast$-functors}
\label{sec:univ}
In this section we construct the $\ast$-functors which will allows us to prove Theorem~\ref{thm:main1}. 

\begin{remark}
\label{rem:explanation_U}
Each construction will be an instantiation of the universal $\Cstar$-category $\mathrm{U}(Q,R)$ associated to a quiver~$Q$ with an admissible set of relations~$R$. Concretely, $\mathrm{U}(Q,R)$ is the universal ($=$~initial) $\Cstar$-category~$A$ endowed with a representation of~$(Q,R)$, \ie\ $A$ comes equipped with a quiver morphism $\rho: Q\to A$ such that the images of the generating arrows of~$Q$ satisfy in~$A$ the conditions specified in~$R$. 
Given a pair $(Q,R)$, the existence of $\mathrm{U}(Q,R)$ is not automatic but is guaranteed if a few conditions are satisfied (the ``admissible'' part). If $R$ consists only of $*$-algebraic equations, as it will always be the case in this section, then it suffices that for each arrow~$a$ in~$Q$ there exists a uniform bound  for the norm of $\rho(a)$ over all representations~$\rho$ of $(Q,R)$ in $\Cstar$-categories. 
The reader is referred to \cite{ivo:unitary}*{\S 2} for further details.
\end{remark}
Throughout this section we fix a positive integer $n\geq 1$.
\begin{defi}[Universal direct sum $S(n)$]
\label{defi:S}
Let $S(n)$ be the universal $\Cstar$-category containing~$n$ objects and a direct sum for them.
More precisely, $S(n):=\mathrm{U}(Q,R)$ for the quiver~$Q$ with objects $o_1,\ldots,o_n$ and $s(n)$, with arrows $v_i:o_i\to s(n)$ ($i=1,\ldots,n$), and with the following set of $\ast$-algebraic relations:
\[
R=\left\{ v_1v_1^*+\ldots + v_nv_n^*=1_{s(n)}\,,\quad v_i^*v_j = \delta_{ij}   \; (i,j=1,\ldots,n)  
  \right\} \;  .
\]
\end{defi}


\begin{notation}
Let $\bbF^n:= \bbF\sqcup \cdots \sqcup \bbF$ be the coproduct in $\Cstarcatun$ of $n$ copies of~$\F$. Following Example \ref{ex:C*cats}, 
its objects will be denoted by $\bullet_1, \cdots, \bullet_n$.
\end{notation}

\begin{defi}
\label{defi:S_n}
Let $S_n\colon \bbF^n \to S(n)$ be the unique $*$-functor sending $\bullet_i$ to $o_i$, for $i=1,\ldots, n$.
\end{defi}

\begin{lemma}
\label{lemma:S_n}
The $*$-functor $S_n\colon \F^n\to S(n)$ is isomorphic to the restriction of $\sigma_{\F^n}\colon \F^n \to (\F^n)_\oplus$ to the full $\Cstar$-subcategory of $(\F^n)_\oplus$ with objects $\sigma_{\F^n}(\bullet_1),\ldots,\sigma_{\F^n}(\bullet_n)$ and 
$\sigma_{\F^n}(\bullet_1)\oplus\cdots\oplus\sigma_{\F^n}(\bullet_n)$.
\end{lemma}

\begin{proof}
The proof is a straightforward verification that we leave for the reader.
\end{proof}

Whenever convenient, we will freely identify $S_n$ as in Lemma \ref{lemma:S_n}. 

\begin{defi}[Universal projection matrix $P(n)$]
\label{defi:P(n)}
Let $P(n)$ be the universal $\Cstar$-category containing $n$ objects and an idempotent self-adjoint $n\times n$ matrix of arrows between them. More precisely, consider the quiver~$Q$ with~$n$ vertices $o_1,\ldots, o_n$, with arrows $p_{ij}\colon o_j\to o_i$ (for $i,j\in\{1,\ldots,n\}$), and with the following set of $*$-algebraic relations:
\[
R=\{ [p_{ij}]^*=[p_{ij}] = [p_{ij}]^2 \} = \left\{ p_{ji}^* = p_{ij} = \sum_{k=1}^n p_{ik} p_{kj} \quad ( i,j=1,\ldots,n  ) \right\} \;.
\]
Then we set $P(n):= \mathrm{U}(Q,R)$.
This universal $\Cstar$-category exists, because the above equations~$R$ force the norm of each $p_{ij}$ to be at most equal to one.
\end{defi}

\begin{remark}
\label{remark:inspiration_for_defs}
Let $[p_{ij}]$ be a projection matrix on objects $x_1,\ldots,x_n$ in a $\Cstar$-category $A$. We say that an object $r\in \obj(A)$ is a \emph{range object for~$[p_{ij}]$} if $\iota_A(r)$ is a range object for the projection $p:=[p_{ij}]$ on the object $x_1\cdots x_n = \iota_A(x_1)\oplus \cdots\oplus \iota_A(x_n)$ of $A^\natural_\oplus$.
This is precisely the case if there exists some isometry $v\colon r\to x_1\cdots x_n$ with $vv^*=p$.
Since its target object is a direct sum, $v$ is determined by its components $s_i\colon x_i\to r$, as in $v=:[s_1 \,\ldots\, s_n]^*$. They are such that $p_{ij}=s^*_is_j$: 
\[
[p_{ij}] 
= p
= vv^*
= 
\left[ \begin{array}{c}
s^*_1 \\
\vdots \\
s^*_n
\end{array}\right]
\left[ \begin{array}{ccc}
s_1 &
\cdots &
s_n
\end{array}\right]
= [s^*_is_j]
\;.
\]
It is straightforward to verify that (with these notations and for any choice of such a range object~$r$) the following three sets of relations are equivalent:
\begin{enumerate}
\item[(i)] $p^*=p=p^2$ (with $p=vv^*$ or $p_{ij}=s^*_is_j$);
\item[(ii)] $1_r = v^*v$  (with $v=[s_1 \,\ldots\, s_n]^*$);
\item[(iii)] $1_{r}=\sum_{k=1}^{n} s_ks_k^*$.
\end{enumerate}
If we also take into consideration the direct sum $x:=x_1\cdots x_n$ with its isometries $v_i\colon x_i\to x$, we may visualize (most of) the relations between all these maps in the following commutative diagram:
\begin{equation}\label{eq:master_diagram}
\xymatrix{
x_j \ar[rd]^{s_j} \ar[rr]^-{p_{ij}} \ar[dd]_{v_j} &&  x_i \\
  & r \ar[ur]^{s_i^*} \ar[rd]_-{v} &   \\
x \ar[ru]_-{v^*} \ar[rr]_-{p} && x \ar[uu]_{v_i^*}
}
\end{equation}
Note that -- just as in (i) when formulated in terms of the matrix $[p_{ij}]$ as in Definition~\ref{defi:P(n)} -- the relations in~(iii) can be expressed entirely inside~$A$. 
This motivates the following crucial definition.
\end{remark}

\begin{defi}[Universal projection matrix with range $R(n)$]
\label{defi:R(n)}
Let $R(n)$ be the universal $\Cstar$-category with $n$ objects, a projection matrix of arrows among them, and a range object for the latter. 
More precisely, consider the quiver~$Q$ with vertices
\[
\obj(R(n)):= \{ o_1, \ldots, o_n, r(n) \}\,,
\]
with $n$ arrows $s_i\colon o_i\to r(n)$ ($i=1,\ldots,n$), and with the following $\ast$-algebraic relation:
\[
R= \left\{ 1_{r(n)}=\sum_{k=1}^{n} s_ks_k^*  \right\} \;. 
\]
Then we set $R(n):= \mathrm{U}(Q,R)$. The above equation forces the norm of each $s_i$ to be at most equal to one and so the $\Cstar$-category $R(n)$ is well-defined. 
\end{defi}

\begin{defi}[Universal sum of objects with a projection $SP(n)$]
Let $SP(n)$ be the universal $\Cstar$-category containing $n$ objects, a sum of these objects, and a projection on the sum. 
More precisely, consider the quiver $Q$ with vertex set 
\[
\obj(SP(n)) := \{o_1, \ldots ,o_n , s(n) \}\,,
\]
with $n+1$ arrows $v_i\colon o_i \to s(n)$ ($i=1,\ldots,n$) and $p \colon s(n) \to s(n)$, and with the following set of $\ast$-algebraic relations:
\[
R= \left\{ 
1_{s(n)} = \sum_{k=1}^n v_kv_k^*\,,\quad v_i^*v_j = \delta_{ij}  \; (i,j=1,\ldots,n) \,, \quad p^*=p=p^2  
\right\}
\;.
\]
Then we set $SP(n):= \mathrm{U}(Q,R)$. Since the above projection and isometry relations force the norm of~$p$ and $v_i$ to be at most equal to one, the $\Cstar$-category $SP(n)$ is well-defined.
\end{defi}

\begin{defi}[Universal sum of objects with a retract $SR(n)$]
\label{defi:SD}
Let $SR(n)$ be the universal $\Cstar$-category containing $n$ objects, a sum of these objects, and a retract of the sum. More precisely, consider the quiver $Q$ with vertex set 
\[
\obj(SR(n)) := \{o_1, \ldots ,o_n , s(n) , r(n) \}\,,
\]
with $n+1$ arrows $v_i\colon o_i \to s(n)$ ($i=1,\ldots,n$) and $v \colon r(n) \to s(n)$, and with the following set of $\ast$-algebraic relations:
\[
R=\left\{
1_{s(n)}  = \sum_{k=1}^n v_kv_k^*\,,\quad v_i^*v_j = \delta_{ij}  \; (i,j=1,\ldots,n) \,, \quad v^*v=1_{r(n)}  
\right\}
\;.
\]
Then we set $SR(n):= \mathrm{U}(Q,R)$. These isometry relations force the norm of $v$ and $v_i$ to be at most equal to one, and so the $\Cstar$-category is well-defined.
\end{defi}

For the reader's convenience and in order to simplify the proofs in the sequel, we now spell out explicitly the universal properties of the above $\Cstar$-categories.
\begin{lemma}\label{lem:char}
Let $n\geq 1$ be a positive integer and $B$ a $\Cstar$-category.
\begin{enumerate}
\item[(i)] A $*$-functor $F\colon \F^n\to B$ is uniquely determined by the choice of~$n$ objects $x_i=F(\bullet_i)\in \obj(B)$ \textup($i=1,\ldots,n$\textup).

\item[(ii)] A $*$-functor $F\colon S(n)\to B$ is uniquely determined by the choice of~$n$ objects $x_i=F(o_i)\in \obj(B)$ \textup($i=1,\ldots,n$\textup) and of a direct sum $x=F(s(n))$ of these objects in~$B$ (with specified isometries $F(v_i) \in B(x_i, x)$).

\item[(iii)] A $*$-functor $F\colon P(n)\to B$ is uniquely determined by the choice of~$n$ objects $x_i=F(o_i)\in \obj(B)$ \textup($i=1,\ldots,n$\textup) and of a projection matrix $[Fp_{ij}]$ on them (as in Remark \ref{remark:inspiration_for_defs}). 

\item[(iv)] A $*$-functor $F\colon SP(n)\to B$ is uniquely determined by the choice of~$n$ objects $x_i=F(o_i)\in \obj(B)$ \textup($i=1,\ldots,n$\textup), a direct sum $x=F(s(n))$ for them (with specified isometries $F(v_i) \in B(x_i, x)$), and a projection~$F(p)\in B(x,x)$ on~$x$.

\item[(v)] A $*$-functor $F\colon R(n)\to B$ is uniquely determined by the choice of $n+1$ objects $x=F(o_i)$ \textup($i=1,\ldots,n$\textup) and $r=F(r(n))$ of~$B$ together with an isometry $v=[Fs_1 \,\cdots\, Fs_n]^*\colon r\to x_1\oplus \cdots \oplus x_n$ in $B_\oplus$. (In the terminology of Remark~\ref{remark:inspiration_for_defs}, this amounts to choosing arrows $F(s_i)\in B( Fx_i, Fr(n))$ such that $[F(s_i)^*F(s_j)]$ is a projection matrix with range object $F(r(n))$).

\item[(vi)] A $*$-functor $F\colon SR(n)\to B$ is uniquely determined by the choice of  $n+1$ objects $x_i=F(o_i)\in \obj(B)$ \textup($i=1,\ldots,n$\textup) and $r=F(r(n))$ of~$B$, of a direct sum $x=F(s(n))$ of these objects \textup(with isometries $Fv_i\in B(x_i,x)$\textup), and of a retract $r=F(r(n))$ of~$x$ (with isometry $Fv\in B(r,x)$).

\end{enumerate}
\end{lemma}
\begin{proof}
As already suggested in Remark~\ref{rem:explanation_U}, in each case the claim follows automatically from the description of the generating quiver with relations $(Q,R)$ and from the universal property of the associated $\Cstar$-category $\mathrm U(Q,R)$; see \cite{ivo:unitary}*{Theorem~2.3}.
\end{proof}

\begin{defi}
\label{defi:R_n}
By comparing the relations imposed on their generating arrows, the universal properties of the above-constructed $\Cstar$-categories immediately induce (as in Lemma~\ref{lem:char}) the following commutative diagram of $*$-functors.
\[
\xymatrix@C=0pt@R=9pt{
&& \F^n \ar[dl] \ar[dr]^{S_n} & \\
& P(n) \ar[dl]_{R_n} \ar[dr] && S(n) \ar[dl] \\
R(n) \ar[dr] && SP(n) \ar[dl] & \\
& SR(n) &&
}
\]
The functor $S_n$ was already defined (see Definition~\ref{defi:S_n}). The remaining ones are determined by the following assignments, with $i,j=1,\ldots,n$ (\cf\ Remark\,\ref{remark:inspiration_for_defs}):
$$
\begin{array}{lcl}
\F^n \to P(n), \bullet_i \mapsto o_i && R_n\colon P(n) \to R(n), p_{ij} \mapsto s_i^*s_j\\
P(n) \to SP(n), p_{ij} \mapsto v^*_i p v_j && S(n) \to SP(n), v_i \mapsto v_i\\
SP(n) \to SR(n), v_i \mapsto v_i \textrm{ and } p \mapsto vv^* && R(n) \to SR(n), s_i \mapsto v^*v_i \,.\\
\end{array}
$$
For $n=0$, we also define $R_0:P(0)\to R(0)$ to be the unique $*$-functor $\emptyset\to \mathbf0$ from the initial object to the final object of $\Cstarcatun$, \ie\ from the empty $\Cstar$-category to the $\Cstar$-category with a unique object whose endomorphism ring is zero. Note that this is consistent with our notation since for $n=0$ the equation $1_{r(0)}=\sum_{\emptyset} s_ks^*_k =0$ says that the unique object $r(0)$ of $R(0)$ is a zero object.
\end{defi}

\begin{lemma}
\label{lemma:ff_comparisons}
All the $*$-functors introduced in Definition~\ref{defi:R_n}, except for the two families $\F^n\to P(n)$ and $S(n)\to SP(n)$ ($n\geq1$), are fully faithful.
\end{lemma}

\begin{proof}
The claim is obvious for $S_n$ by Lemma \ref{lemma:S_n}, and for $R_0$ because it is empty.
The two exceptional families are seen \emph{not} to be fully faithful, simply because there do exist non-identity projection matrices on direct sums (\eg\ in $\Hilb$). 
Since the proof is similar for the remaining four families, we will restrict ourselves to the $\ast$-functor $R_n$ ($n\geq1$). 
Consider the canonical $*$-functor $\iota_{P(n)}\colon P(n)\to P(n)_\oplus^\natural$. 
By Lemma~\ref{lem:char}(iii) it is determined by the image of the generating arrows $p_{ij}\colon x_j\to x_i$. Hence, in order to extend $\iota_{P(n)}$ along~$R_n$ as in the following commutative diagram
\[
\xymatrix{
P(n) \ar[d]_{R_n} \ar[r]^-{\iota_{P(n)}} & P(n)_\oplus^\natural  \\
R(n) \ar@{..>}[ru]_-{F}  & \,,
} 
\]
it suffices by Lemma~\ref{lem:char}(v) to find  in $P(n)_\oplus^\natural$ an object~$r$ and arrows $s_i\colon x_i\to r$ such that $s_i^*s_j=p_{ij}$. We have a canonical choice where we set~$r$ to be $(x_1 \cdots x_n , [p_{ij}])$ and~$s_i$ to be the composition $x_i\to x_1 \cdots x_n \to r$ of the isometry of the direct summand~$x_i$ with the coisometry of the direct summand~$r$.
Thus $F\circ R_n= \iota_{P(n)}$. Since $\iota_{P(n)}$ is faithful so is $R_n$.

It remains to prove that $R_n$ is full. To this end, recall from \cite{ivo:unitary}*{\S 2} that for any quiver with admissible relations $(Q,R)$ there is a free $*$-category $\mathrm FQ$ on the quiver~$Q$ and thus an induced $*$-functor $\pi\colon \mathrm FQ\to \mathrm U(Q,R)$; see \cite{ivo:unitary}*{Constr.~2.4}. Moreover, if the relations $R$ are algebraic (such as the ones defining~$R(n)$) then the image $B:=\mathrm FQ/\ker(\pi)$ in $\mathrm U(Q,R)$ is a dense $*$-subcategory, \ie, each Hom-set of~$B$ is  norm-dense; see \cite{ivo:unitary}*{Example~2.7}. 
In particular -- applying this to the quiver with relations in Definition~\ref{defi:R(n)} -- for every arrow $b\in R(n)(o_j,o_i)$ there is a sequence $(b_k)_k\subset R(n)(o_j,o_i)$ converging in norm to~$b$, where each $b_k$ belongs to~$B$, \ie\ $b_k$ is a (finite) $*$-algebraic combination of arrows of~$Q$. But, since $s_i^*s_j=R_n(p_{ij})$, every such combination with domain~$o_j$ and codomain~$o_i$ can already be written using the generating arrows~$p_{ij}$ of~$P(n)$, hence each $b_k$ also defines an element in~$P(n)$. 
Moreover, we have just proved that $R_n\colon P(n)\to R(n)$ is faithful, \ie\ that it is an isometric inclusion. Therefore the expression $\lim_k b_k$ defines an arrow $a\in P(n)(o_j,o_i)$ such that $R_n(a)=b$. We conclude that $R_n$ is full, as claimed.
\end{proof}

\section{The Morita model structure}

In this section we start by constructing the \emph{Morita model structure} on the category of $\Cstar$-categories; see Theorem~\ref{thm:morita_model}. Then we describe its fibrant objects and the Hom-sets of the associated homotopy category; see Propositions~\ref{prop:fibrant} and \ref{prop:morphisms_Morita} and Corollary \ref{cor:new}. Finally, we relate Morita equivalence of $\Cstar$-categories with Morita-Rieffel equivalence of $\Cstar$-algebras; see Proposition~\ref{prop:Morita_agreement}.

\subsection{The unitary model} \label{subsec:unitary}

Recall from \cite{ivo:unitary}*{\S3} that the unitary model structure on~$\Cstarcatun$ consists of the following three classes of $*$-functors:
\begin{align*}
\Weq_\uni &= \{\textrm{unitary equivalences (see Definition~\ref{def:unitary})} \} \\
\Cof_\uni &= \{\ast\text{-}\textrm{functors } F \textrm{ such that } \obj(F) \textrm{ is injective}\} \\
\Fib_\uni & = \{\ast\text{-}\textrm{functors } F \textrm{ allowing the lift of unitaries of the form } Fx\stackrel{\sim}{\to} y \}\,.
\end{align*}
Recall also that $\mathbf I$ denotes the $\Cstar$-category with two objects $0$ and~$1$ and with a unitary isomorphism $u:0\stackrel{\sim}{\to} 1$. Each Hom-space of $\mathbf I$ is therefore one dimensional. 
Following \cite{ivo:unitary} we will name a $*$-functor $\F\to A$ by the object $x=F(\bullet)\in \obj(A)$ which determines it uniquely. The unitary model structure is cofibrantly generated  and its sets of generating (trivial) cofibrations are the following:
\begin{eqnarray*}
I_\uni = \{ U\colon \emptyset \to \F , V\colon \F\sqcup\F\to \mathbf1 , W\colon P\to \mathbf1 \}  &&
J_\uni = \{0\colon \F\to \mathbf I\}\,.
\end{eqnarray*}

\begin{cor}
\label{cor:pushout_unitrcof}
The class of unitary equivalences which are injective on objects is closed under pushouts in~$\Cstarcatun$.
\qed 
\end{cor}
\begin{proof}
This is an immediate consequence of the existence of the unitary model, since in any model category the class of trivial cofibrations is closed under pushouts.
\end{proof}

In the particular case of the generating unitary trivial cofibration $0:\F\to \mathbf I$, the pushouts of Corollary~\ref{cor:pushout_unitrcof} admit the following description: 
\begin{lemma}
\label{lemma:0_pushout}
Given a $\Cstar$-category $A$ and an object $x_0\in \obj(A)$, the pushout
\begin{equation}\label{eq:0_pushout}
\xymatrix{
\F \ar[r]^-{x_0} \ar[d]_0 \ar@{}[dr]|{\lrcorner}& A \ar[d] \\
\mathbf I \ar[r] & A\sqcup_\F \mathbf I
}
\end{equation}
has the following description: $A\sqcup_\F\mathbf I$ is the $\Cstar$-category consisting of $A$ plus an extra object~$x_1$, for which we set $(A\sqcup_\F\mathbf I)(x_0,x_1):= u\circ A(x_0,x_0)$, $(A\sqcup_\F\mathbf I)(x_1,x_0):=A(x_0,x_0)\circ u^*$ and $(A\sqcup_\F \mathbf I)(x_1,x_1):= u\circ A(x_0,x_0)\circ u^*$, with the new relations $u^*u= 1_{x_0}$ and $uu^*=1_{x_1}$.
Then the $*$-functor $A\to A\sqcup_\F \mathbf I$ is the inclusion and the $*$-functor $\mathbf I\to A\sqcup_\F \mathbf I$ maps the generating unitary morphism of~$\mathbf I$ to $u:x_0\to x_1$. 
\end{lemma}

\begin{proof}
For any object $x\in \obj (A)$, composition with the new morphism~$u$ induces isometric isomorphisms of Banach spaces 
$(A\sqcup_\F\mathbf I)(x,x_0)\stackrel{\sim}{\to} (A\sqcup_\F\mathbf I)(x,x_1)$ and
$(A\sqcup_\F\mathbf I)(x_1,x)\stackrel{\sim}{\to} (A\sqcup_\F\mathbf I)(x_0,x)$. Similarly  
$(A\sqcup_\F\mathbf I)(x_1,x_1)\stackrel{\sim}{\to} (A\sqcup_\F\mathbf I)(x_0,x_0)$
by conjugation with~$u$.
Thus all the Hom-spaces of $A\sqcup_\F\mathbf I$ are completely determined. 
Moreover, its composition and involution are forced by those of~$A$. 
It is now evident that this data endows $A\sqcup_\F\mathbf I$ with a well-defined $\Cstar$-category structure satisfying the required pushout property.
\end{proof}


\subsection{The Morita model} 
\begin{defi}
\label{defi:Mor}
A $*$-functor $F: A\to B$ is a \emph{Morita equivalence} if the induced $*$-functor $F^\natural_\oplus = (F_\oplus)^\natural \colon A_\oplus^\natural \to B^\natural_\oplus$ is a unitary equivalence; justification for this terminology will be provided in Proposition \ref{prop:Morita_agreement}.
\end{defi}

\begin{remark}
Note that by Proposition~\ref{prop:UP_saturated}(iv) all unitary equivalences are Morita equivalences. 
\end{remark}

\begin{lemma}
\label{lemma:recognition_Moreq}
A $*$-functor $F\colon A\to B$ is a Morita equivalence if and only if:
\begin{itemize}
\item[(i)] it is fully faithful;

\item[(ii)] the smallest full subcategory of $B^\natural_\oplus$ containing $\iota_B(FA)$ and closed under taking isomorphic objects, direct sums, and retracts,  is the whole~$B^\natural_\oplus$. 
\end{itemize}
More interestingly, the same conclusion holds with (ii) replaced by:
\begin{itemize}
\item[(ii)'] the smallest full subcategory of $B$ containing~$FA$ and closed under taking isomorphic objects, and whatever direct sums and range objects of projection matrices (in the sense of Remark~\ref{remark:inspiration_for_defs}) exist in~$B$, is the whole~$B$. 
\end{itemize}
\end{lemma}

\begin{proof}
If $F$ is a Morita equivalence, then (i) holds by Proposition \ref{prop:UP_saturated}(iii) while~(ii) holds because $F^\natural_\oplus$ is  essentially surjective: every object of $B^\natural_\oplus$ is isomorphic to one of the form $F^\natural_\oplus(x_1\cdots x_n, [p_{ij}])=(Fx_1\cdots Fx_n, [Fp_{ij}]))$ with $x_1,\ldots,x_n \in \obj A$, so in particular it is a retract of a direct sum of objects in $\iota_B(FA)$.

Note that if (ii) holds then the full subcategory of $B^\natural_\oplus$ generated by $\iota_B(FA)$ by taking direct sums and retracts contains in particular $\iota_B(B)$. Since the range object of a projection matrix $[p_{ij}]$ in~$B$ is by definition a suitable retract of a direct sum in $B^\natural_\oplus$ (which also happens to belong to~$B$), we conclude that (ii) implies~(ii)'.

Thus it remains to prove that (i) combined with~(ii)' imply that $F^\natural_\oplus$ is a unitary equivalence.
By (i) and Proposition \ref{prop:UP_saturated}(iii) again, it suffices to prove that $F^\natural_\oplus$ is essentially surjective. 
Note that a direct sum $r=x_1\oplus \cdots \oplus x_n$ in~$B$ is a special case of a range object for a projection matrix, namely the identity matrix. 
Thus (ii)' means that every $y\in B$ can be reached from $FA$, in a finite number of steps, by adding all range objects of projection matrices between objects that were produced in the previous steps.
We claim that one single step always suffices. In order to see this, let $r$ be a range object in~$B$ for the projection matrix $p=[p_{ij}]$ on $r_1,\ldots,r_n$, and assume that each $r_k$ is, in its turn, a range object in~$B$ for the projection matrix $p_k=[p_{ij}^k]$ on $x^k_1,\ldots,x^k_{n_k}$ $(k=1,\ldots n)$.
By definition, this means that there exist in~$B^\natural_\oplus$ isometries 
$v\colon \iota_B(r) \to \iota_B(r_1)\cdots \iota_B(r_n)$ and 
$v_k\colon \iota_B(r_k) \to \iota_B(x_1^k)\cdots \iota_B(x_{j_k}^k)$ 
such that $vv^*=p$ and $v_kv_k^*=p_k$. 
Consider the following composition in~$B^\natural_\oplus$
\[
w\quad\colon \quad
\xymatrix{ 
\iota_B(r) \ar[r]^-v &
 \iota_B(r_1)\cdots \iota_B(r_n) \ar[rr]^-{\mathrm{diag}(v_1,\ldots,v_k)}   && z } ,
\]
where $z:=\iota_B(x^1_1)\cdots \iota_B(x^1_{n_1}) \cdots \iota_B(x^k_1)\cdots \iota_B(x^k_{n_k})$ and where $\mathrm{diag}$ denotes the evident block-diagonal matrix.
Then $v$ is an isometry:
\[
w^*w 
= v^*\mathrm{diag}(v_1,\ldots,v_k)^* \mathrm{diag}(v_1,\ldots,v_k) v 
= v^* \mathrm{diag}(1_{r_1},\ldots , 1_{r_n}) v = 1_r 
\;,
\]
and therefore $ww^*$ is a projection matrix in~$B$ on  $x^1_1, \ldots , x^1_{n_1} , \ldots , x^k_1,\ldots, x^k_{n_k}$ for which $r$ is a range object. Hence, by an easy recursion (and hypothesis~(ii)'), every $y\in \obj B$ is a range object for a projection matrix $[q_{ij}]$ on some objects $Fx_1,\ldots,Fx_n$ in the image of~$F$, as claimed. 
Since $F$ is fully faithful, also the matrix comes from~$A$, say $[q_{ij}]=[Fp_{ij}]$, and therefore by the uniqueness of range objects $\iota_B(y)$ must be (unitarily) isomorphic to $F^\natural_\oplus(x_1\cdots x_n, [p_{ij}])$. Hence $\iota_B(B)$ is contained in the essential image of~$F$, and since $F$ is fully faithful this implies that $F^\natural_\oplus$ is (unitarily) essentially surjective.
\end{proof}

\begin{defi} 
\label{defi:IJ_Mor}
The sets of generating cofibrations and generating trivial cofibrations of the Morita model structure are the following: 
\begin{eqnarray*}
I_\Mor :=  I_\uni &&
J_\Mor := \{R_n: P(n) \to R(n) \mid n\geq0 \} \,,
\end{eqnarray*}
where the $*$-functors $R_n$ are as in  Definition \ref{defi:R_n}.
\end{defi}

\begin{remark}
It follows from Lemmas \ref{lemma:ff_comparisons} and \ref{lemma:recognition_Moreq} that the $*$-functors $R_n$, $n\geq1$, are Morita equivalences. Indeed, $R_n$ is fully faithful and the only missed object is by the very construction the range object of a projection matrix in the image. 
Note that $R_0\colon \emptyset\to \mathbf0$ is also a Morita equivalence since every object in $\mathbf 0^\natural_\oplus$ is isomorphic to~$0$ and $(R_0)^\natural_\oplus$ is the $\ast$-functor sending the unique object of $\emptyset^\natural_\oplus\simeq\mathbf0$ to $0\in \mathbf0^\natural_\oplus$.
\end{remark}

\begin{thm}[The Morita model]
\label{thm:morita_model}
The category $\Cstarcatun$ admits a cofibrantly generated model structure where the weak equivalences are the Morita equivalences (see Definition~\ref{defi:Mor}) and the cofibrations the $*$-functors which are injective on objects. 
Moreover, $I_\Mor$ and $J_\Mor$ (see Definition~\ref{defi:IJ_Mor}) are generating sets of cofibrations and trivial cofibrations, respectively. 
\end{thm}

\begin{notation}
In what follows we will write $\Weq_\Mor$, $\Cof_\Mor =\Cof_\uni=:\Cof$, and $\Fib_\Mor$ for the classes of weak equivalences, cofibrations and fibrations of the Morita model.
Moreover, we will denote by $\cM_\Mor$ the Morita model category  and by $\Ho(\cM_\Mor):= \Cstarcatun[\Weq_\Mor^{-1}]$ its homotopy category, obtained by formally inverting all Morita equivalences.
We will also freely use the standard notations $J\textrm{-cell}, I\textrm{-cof}$, $J\textrm{-inj}$, etc.\ from model category theory; see \cite{hovey:model}*{\S 2.1.2}.
\end{notation}

\subsection{Proof of Theorem~\ref{thm:morita_model}}

The proof will consist in verifying the six conditions (i)-(vi)\ of \cite[Thm.~2.1.19]{hovey:model}. 
Condition (i) follows immediately from the functoriality of $(-)_\oplus^\natural$ and from the fact that the class of unitary equivalences satisfies the corresponding condition (i). The smallness conditions (ii)-(iii) follow immediately from the fact that the domain of each element of $I_\Mor$ and $J_\Mor$ is a $\Cstar$-category generated by finitely many morphisms.

\begin{notation}
Let $\mathbf{Surj}$ be the class of $\ast$-functors $F:A \rightarrow B$ such that the associated maps $A(x,x') \stackrel{\sim}{\rightarrow} B(Fx,Fx')$ are bijections and the associated map $\obj(A) \twoheadrightarrow \obj(B)$ is surjective.
\end{notation}

\begin{remark}
\label{remark:Surj_trfib}
Note that by \cite{ivo:unitary}*{Corollary 3.8}, $\mathbf{Surj}$ is precisely the class of trivial fibrations of the unitary model structure.
\end{remark}

The proof of the remaining three conditions (iv)-(vi) of \cite[Thm.~2.1.19]{hovey:model} reduces to the following two claims:
\begin{align}
\label{eq:former} & J_\Mor\text{-}\mathrm{inj} \cap \Weq_\Mor= \mathbf{Surj} \\
\label{eq:later} & J_\Mor\text{-}\mathrm{cell} \subset \Weq_\Mor\,.
\end{align}
Indeed, by \cite[Prop.~3.18]{ivo:unitary} we have $J_\Mor \subset \Cof_\uni = I_\uni \textrm{-cof}$ which  implies that $J_\Mor\textrm{-cell}\subset I_\uni\textrm{-cof}$. By combining it with~\eqref{eq:later} we obtain condition~(iv).
From Remark \ref{remark:Surj_trfib} and \eqref{eq:former} we obtain the equalities $I_\uni\textrm{-inj}= \Weq_\uni\cap \Fib_\uni = \mathbf{Surj} = J_\Mor\textrm{-inj}\cap  \Weq_\Mor $, and hence condition~(v). 
Finally, \eqref{eq:former} and Remark \ref{remark:Surj_trfib} imply the second choice for condition~(vi), \ie\ the inclusion 
$J_\Mor\textrm{-inj}\cap \Weq_\Mor =\mathbf{Surj}= \Weq_\uni \cap \Fib_\uni \subset I_\uni\textrm{-inj} $.\\

Therefore it remains only to prove \eqref{eq:former} and \eqref{eq:later}. Let us start with the latter.
For this, consider for every $n\geq1$ the following pushout
\begin{equation}
\label{eq:little_pushouts}
\xymatrix{
 \F \ar[d]_0 \ar[r]^-{r(n)} \ar@{}[dr]|{\lrcorner} & SR(n) \ar[d]^{V_0} \\
 \mathbf I \ar[r] & SR(n)\sqcup_\F \mathbf I
}
\end{equation}
as in Lemma~\ref{lemma:0_pushout}, where $V_0$ denotes the resulting fully faithful inclusion.
Let us also denote by $r_0(n)$ and $r_1(n)$ the images in $SR(n)\sqcup_\F \mathbf I$ of the two objects $0,1$ of~$\mathbf I$. Note that $r_1(n)$ is unitarily isomorphic to $r_0(n)= V_0(r(n))$ and so it is a retract of~$V_0(s(n))$. 
Thus, there exists a unique $*$-functor $SR(n)\to SR(n)\sqcup_\F\mathbf I$ agreeing with~$V_0$ on $o_1,\ldots,o_n,s(n)$ and sending $r(n)$ to~$r_1(n)$. After precomposing it with the canonical $*$-functor $R(n)\to SR(n)$, we obtain a well-defined $*$-functor $V_1\colon R(n)\to SR(n)\sqcup_\F\mathbf I$ agreeing with $V_0$ over $P(n)$ and sending $r(n)$ to~$r_1(n)$. This gives rise to the following commutative square
\begin{equation}\label{eq:square}
\xymatrix{
 P(n) \ar[d]_-{R_n} \ar[r]  & SR(n) \ar[d]^{V_0} \\
 R(n) \ar[r]_-{V_1} & SR(n) \sqcup_\F{\bf I}\,.
}
\end{equation}
\begin{lemma}
\label{lemma:little_pushouts}
For every $n\geq1$, the above commutative square \eqref{eq:square} is a pushout. Moreover, each $*$-functor in it is fully faithful.
\end{lemma}

\begin{proof}
By symmetry, we observe that there is an automorphism $\Psi\colon SR(n) \sqcup_\F{\bf I}\stackrel{\sim}{\to}SR(n) \sqcup_\F{\bf I}$ fixing $SP(n)$ (and thus also~$P(n)$) and exchanging $r_0(n)$ with $r_1(n)$, so that $(\Psi\circ V_0)|_{R(n)}=V_1$.
In particular $V_1$ is fully faithful. Since the remaining $\ast$-functors are fully faithful, we conclude the second claim of the lemma. Now, let us show that \eqref{eq:square} is a pushout square. Consider two $*$-functors $T_0,T_1$ as in the following commutative diagram (without~$T$):
\[
\xymatrix{
P(n) \ar[r] \ar[d]_{R_n} & SR(n) \ar[d]^-{V_0} \ar@/^3ex/[ddr]^{T_0} & \\
R(n) \ar@/_3ex/[drr]_{T_1} \ar[r]_-{V_1} & SR(n)\sqcup_\F \mathbf I  \ar@{..>}[dr]|T & \\
&& C\,.
}
\]
By commutativity, the objects $T_0(r(n))$ and $T_1(r(n))$ of~$C$ are direct summands of $T_0(s(n))$ with the same range projection $T_0(vv^*)= T_0([p_{ij}])= T_1([s^*_is_j])$ on $T_0(s(n))$. Hence there exists a unique compatible unitary isomorphism $v$ from $T_0(r(n))$ to $T_1(r(n))$. It is now clear that there exists a unique $*$-functor $T$ mapping $u\colon r_0(n)\to r_1(n)$ to~$v$ and making the whole diagram commute.
\end{proof}

\begin{prop}\label{prop:aux}
$J_\Mor\text{-}\mathrm{cell} \subset \Weq_\Mor$.
\end{prop}

\begin{proof}
Note first that, since the class of unitary equivalences is stable under transfinite compositions and since the functor $(-)_\oplus^\natural$ preserves transfinite compositions, the class of Morita equivalences is also stable under transfinite compositions. 
Hence, given a $\Cstar$-category $A$ and a $*$-functor $F\colon \smash{P(n) \to A}$, it suffices to show that the $\ast$-functor $U$ in the following pushout square (for $n\geq0$)
\begin{equation}\label{eq:squares}
\xymatrix{
 P(n) \ar[d]_{R_n} \ar[r]^-F  \ar@{}[dr]|{\lrcorner} & A \ar[d]^{U} \\
 R(n) \ar[r] & B
}
\end{equation}
is a Morita equivalence. This is obvious for $n=0$, as in this case $B=A\sqcup \mathbf0$ is just the result of adding a disjoint zero object. 
In order to prove the case $n >0$ we will provide a concrete description of the $\Cstar$-category~$B$.
Namely, let $B$ be the full subcategory of $A^\natural_\oplus$ with object set $\iota(\obj(A))\cup\{ r:=(x,p) \}$, where $x=F(o_1)\oplus\cdots\oplus F(o_n)$, and $p=[Fp_{ij}]$. The $\ast$-functor $U$ is then the restriction of $\iota_A: A\to A_\oplus^\natural$ to~$B$, and $R(n)\to B$ is the $*$-functor uniquely determined by the assignments $o_i\mapsto F(o_i)$ and $r(n)\mapsto r$.
In order to verify the pushout property, consider two $*$-functors $T_0,T_1$ as in the following commutative diagram (without~$T$):
\[
\xymatrix{
P(n) \ar[r]^-F \ar[d]_{R_n} & A  \ar[d]^-{U} \ar@/^3ex/[ddr]^{T_0}  \\
R(n) \ar@/_3ex/[drr]_{T_1} \ar[r]^-{} & B  \ar@{..>}[dr]|T & \\
&& C \,.
}
\]
Extend $T_0$ along $U$ to a $*$-functor $T:B\to C$ by setting $T(r):= T_1(r(n))$. 
In order to define~$T$ on morphisms, it suffices to define it on the Hom-spaces
\begin{align*}
B(r,y) &= A^\natural_\oplus ((x,p),(y,1_y)) =  \Big(\bigoplus_{j}A(Fo_j,y)\Big) \circ p
\quad \subset \quad \bigoplus_{j}A(Fo_j,y) \\
B(y,r) &= A^\natural_\oplus ((y,1_y),(x,p)) = p\circ\Big(\bigoplus_{i}A(y,Fo_i)\Big)
 \quad \subset \quad \bigoplus_{i}A(y,Fo_i)
\end{align*}
for $y\in \obj(A)=\obj(B)\smallsetminus \{r\}$, as well as on the Hom-space
\[
B(r,r) = A_\oplus^\natural ((x,p),(x,p)) = p\circ \left( \bigoplus_{i,j} A(Fo_j,Fo_i) \right) \circ p 
\; \subset \; \bigoplus_{i,j} A(Fo_j,Fo_i)
\,.
\]
Given such a row vector $a = [a_j] = [a_j] [Fp_{ij}] \in B(r,y)$, a column vector $b= [b_i] = [Fp_{ij}] [b_i] \in B(y,r)$, and a square matrix $c = [c_{ij}] = [Fp_{ij}] [c_{ij}] [Fp_{ij}] \in B(r,r)$, we define the following arrows 
\begin{eqnarray}
T(a) &:= &\sum_j T_0(a_j) T_1(s_j^*) \;\in\; C\big(T_1(r(n)), T_0(y)\big) \nonumber\\
T(b) &:= & \sum_i  T_1(s_i)T_0(b_i) \;\in\; C\big(T_0(y), T_1(r(n))\big) \nonumber \\
T(c) &:=& \sum_{i,j} T_1(s_i) T_0(c_{ij}) T_1(s_j^*) \;\in\; C\big(T_1(r(n)), T_1(r(n))\big)\,. \nonumber
\end{eqnarray}
Using the equalities $T_1(s^*_is_j ) = T_1R_n(p_{ij})= T_0F(p_{ij})$ for all $i$ and $j$, it is straightforward to verify that $T$ is a well-defined $*$-functor $B\to C$ restricting to $T_0$ on $A$ and to $T_1$ on~$R(n)$. 
Since all morphisms into or out of a retract are determined by their components (\cf\ Remark \ref{remark:inspiration_for_defs}(ii)), $T$ is the unique such $\ast$-functor. This shows that in the pushout square \eqref{eq:squares} the map $U$ is fully faithful. It is also evident from this explicit construction that the only object in $B$ not lying in the image of~$U$, that is~$r$, is a range object for a projection matrix in the image. Therefore $U$ is a Morita equivalence by Lemma \ref{lemma:recognition_Moreq}, and the proof is finished.
\end{proof}

\begin{prop}
\label{prop:JcapWeq}
$J_\Mor\text{-}\mathrm{inj} \cap \Weq_\Mor \subset \mathbf{Surj}$.
\end{prop}

\begin{proof}
Let $F:A \to B$ be a $\ast$-functor belonging to $J_\Mor\text{-}\mathrm{inj} \cap \Weq_\Mor $. By Proposition~\ref{prop:UP_saturated}(iii) it is fully faithful and so it remains to show that the map $\obj(A) \to \obj(B)$ is surjective. 
Let $y \in\obj( B)$. 
Since $F \in \Weq_\Mor$, $F^\natural_\oplus$ is unitarily essentially surjective and so there exist an object $x \in A_\oplus^\natural$ and a unitary isomorphism $u\colon F^\natural_\oplus(x) \stackrel{\sim}{\to} y$ in~$B^\natural_\oplus$. 
This data allows us to consider the following commutative diagram 
$$
\xymatrix{
P(n) \ar[dd] \ar[r]^-G & A \ar[d]^F \\
& B \ar[d] \\
SR(n) \ar[r]_-H & B_\oplus^\natural\,, 
}
$$
where $G$ and $H$ are determined by~$x$. More precisely, by definition of $A^\natural_\oplus$ we have $x=(x_1\oplus \cdots\oplus x_n,[p_{ij}])$ for some $x_1,\ldots,x_n\in\obj(A)$ and some projection matrix $[p_{ij}\colon x_j\to x_i]$ in~$A$. Thus, $G$ is determined by $[p_{ij}]$, while $H$ sends $s(n)$ to the canonical direct sum $F(x_1)\oplus\cdots\oplus F(x_n)=F^\natural_\oplus(x_1\oplus\cdots\oplus x_n)$ and $r(n)$ to the canonical retract $F_\oplus^\natural(x)=(Fx_1\oplus\cdots\oplus Fx_n, [Fp_{ij}])$ (and of course $v\colon r(n)\to s(n)$ and $v_i\colon o_i\to s(n)$ are sent to the evident isometries in~$B^\natural_\oplus$). 

The unitary isomorphism $u\colon F^\natural_\oplus(x) \stackrel{\sim}{\to} y$ determines then a unique extension of~$H$ along~$V_0$ to a $*$-functor $\overline{H}\colon SR(n) \sqcup_\F {\bf I} \to B_\oplus^\natural$. Consider the following commutative diagram, where the square in the middle is \eqref{eq:square}:
$$\xymatrix@!0 @R=3pc @C=4pc{
 & & P(n) \ar[rr]^-G \ar[dr]_{R_n} \ar[dl] & 
 & A \ar[d]^F \\
\bbF \ar[r]^-{r(n)} \ar[dr]_-0 & SR(n) \ar@{}[d]|{\lrcorner} \ar[dr]_{V_0} \ar[drrr]^>>>>>>{H} & &
 R(n) \ar[r]^-{\widetilde{H}}
 \ar[dl]|\hole_<<{V_1} \ar@{.>}[ur]^-{\widetilde{\widetilde{H}}}  & B \ar[d] \\
 & {\bf I} \ar[r] & SR(n) \sqcup_{\bbF} {\bf I} \ar[rr]_-{\overline{H}} &  & B_\oplus^\natural \,.
}
$$
By precomposing $\overline{H}$ with the inclusion $V_1$ we obtain a functor $\widetilde{H}$ with values in $B$ (because it maps $r(n)$ to~$y$). Since $F \in J_\Mor\text{-}\mathrm{inj}$, there exists a $\ast$-functor $\widetilde{\widetilde{H}}$ as in the above diagram and so we conclude that there exists an object $x' \in \obj(A)$ such that $F(x')=y$. This shows that $F\in \mathbf{Surj}$, and so the proof is finished.
\end{proof}

 

\begin{prop}\label{prop:aux1}
$\mathbf{Surj} \subset J_\Mor\text{-}\mathrm{inj} \cap \Weq_\Mor$.
\end{prop}
\begin{proof}
Let $F\colon A\to B$ be a $\ast$-functor belonging to $\mathbf{Surj}$.
Since $F$ is by definition a unitary equivalence and thus a Morita equivalence, it suffices to verify that it has the right lifting property with respect to the $\ast$-functors $R_n, n \geq 0$. In the case $n=0$, the requirement is that every zero object of $B$ must come from a zero object of~$A$. This is clearly the case since $F$ is also surjective on objects. For $n\geq1$, consider the following commutative square:
$$
\xymatrix{
P(n) \ar[d]_{R_n} \ar[r]^-G & A \ar[d]^F \\
R(n) \ar[r]_-H & B \,.
}
$$
In order to construct a lift, we will make use of the following commutative diagram:
$$\xymatrix@!0 @R=2pc @C=4pc{
&  & P(n) \ar[rr]^-G \ar[dl]_{} \ar[dd]|\hole & &  *+<1pc>{A}
\ar[dd]^-F \ar[dl] \\
\bbF \ar@{}[ddr]|{\lrcorner} \ar[dd]_0 \ar[r]^-{r(n)} & SR(n) \ar[dd]_{V_0} \ar[rr]^>>>>>>>{\overline{G}}  & & A^\natural_\oplus \ar[dd] & \\
& & R(n) \ar@{..>}[ru]|{L} \ar[rr]|{\hole}^>>>>>>H \ar[dl]_{V_1} &    & *+<1pc>{B} 
\ar[dl]^{\iota_B} \\
{\bf I} \ar[r] & SR(n) \sqcup_{\bbF} {\bf I}
\ar[rr]_-{\overline{H}} &   & B^\natural_\oplus & \,.
}$$
The $\ast$-functor $\overline{G}$ is the extension of $G$ determined by $s(n)\mapsto G(o_1)\oplus \cdots \oplus G(o_n)$ and $r(n)\mapsto (G(o_1)\oplus \cdots \oplus G(o_n), [Gp_{ij}])$, and the $\ast$-functor $\overline{H}:=(\iota_BH , F_\oplus^\natural \overline{G})$ 
is induced by the push-out property proved in Lemma~\ref{lemma:little_pushouts}.
Since $F$ belongs to $\mathbf{Surj}$ and this class is clearly stable under application of the functor $(-)^\natural_\oplus$, the $\ast$-functor $F_\oplus^\natural$ has the right lifting property with respect to $0\colon \bbF \to {\bf I}$. 
Hence, we conclude that $\overline{H}$ lifts to $A^\natural_\oplus$. 
By restricting this lift to $R(n)$ along $V_1$ we obtain a $*$-functor $L$ with $F^\natural_\oplus L(r(n))=\iota_B H(r(n))= \iota_B(Fx)$ for some $x\in \obj(A)$, because $\obj(F)$ is surjective.
By construction, $\iota_A(x)$ is unitarily isomorphic to $\overline{G}(r(n))$ in $A^\natural_\oplus$ over~$P(n)$, \ie\ $x$ is a range object in $A$ for the projection matrix $[Gp_{ij}]$.
Therefore, if we denote by $u\colon \overline{G}(r(n))\stackrel{\sim}{\to} \iota_A(x)$ this unitary morphism, 
we see that the assignments $r(n)\mapsto x$ and $s_i \mapsto u\, \overline{G}(s_i)$ (so that $s^*_is_j\mapsto G(p_{ij})$) define a lift $R(n)\to A$ for the initial square. 
\end{proof}
The proof of Theorem~\ref{thm:morita_model} is now finished since the claims \eqref{eq:former} and \eqref{eq:later} follow respectively from Propositions~\ref{prop:JcapWeq}-\ref{prop:aux1} and Proposition~\ref{prop:aux}.
\subsection{Fibrant objects}

In this subsection we provide a simple characterization of the fibrant objects of the Morita model category; see Proposition~\ref{prop:fibrant}.
\begin{lemma}
\label{lemma:RLP_SvsR}
For every positive integer $n\geq1$ we have $R_n \textrm{-}\mathrm{inj}\subset S_n \textrm{-}\mathrm{inj}$.
\end{lemma}

\begin{proof}
We need to show that every $*$-functor $F\colon A\to B$ having the right lifting property with respect to $R_n$ also has the right lifting property with respect to~$S_n$. Consider the following commutative square of $*$-functors:
\begin{equation}\label{eq:original}
\xymatrix{
\F^n \ar[r]^-G \ar[d]_{S_n} & A \ar[d]^F \\
S(n) \ar[r]_-H & B\,.
}
\end{equation}
Since the identity matrix is a projection matrix, and since a direct sum is a retract of its summands, we can extend it to the (solid) commutative diagram
\[
\xymatrix{
P(n) \ar[d]_{R_n} \ar[r]^-M & \F^n \ar[r]^-G \ar[d] & A \ar[d]^F \\
R(n) \ar@{..>}[urr]^<<<<<<{L} \ar[r]_-N & S(n) \ar@{..>}[ur]_{\widetilde{L}} \ar[r]_-H & B
}
\]
by setting $M(o_i):= \bullet_i$, $M(p_{ij}):= \delta_{ij}$, $N(o_i):=o_i$, $N(r(n)):= s(n)$, and $N(s_i):= v_i$ (for $i,j=1,\ldots,n$).
By hypothesis, there exists a lift $L\colon R(n)\to A$ such that the outer square commutes.
Using this commutativity and the defining relations of our universal $\Cstar$-categories, we conclude that
\[
L(s_i)^*L(s_j)
 = L(s_i^*s_j)
 = LR_n(p_{ij})=\delta_{ij} 
\]
and that
\[
\sum_{k=1}^n L(s_k)L(s_k)^*
 = L\left( \sum_k s_ks_k^* \right) 
 = 1_{L(s(n))} \,.
\]
This shows us that $L(r)$ is a direct sum of the $G(\bullet_i)$'s with isometries $L(s_i)$.
By the definition of $S(n)$, it follows immediately that $L$ lifts along~$N$ to a $\ast$-functor 
$\tilde L\colon S(n)\to A$ with $\tilde L(s(n))=L(r(n))$ and $\tilde L(v_i)=L(s_i)$.
Clearly $F\tilde L=H$ and $\tilde LS_n=G$, which implies that we have found a lift~$\tilde L$ for the original square \eqref{eq:original}. This achieves the proof.
\end{proof}

\begin{prop}\label{prop:fibrant}
For every $\Cstar$-category~$A$ the following conditions are equivalent:
\begin{enumerate}
\item[(i)] The canonical $\ast$-functor $\iota_A\colon A\to A_\oplus^\natural$ is a unitary equivalence.
\item[(ii)]  $A$ is saturated in the sense of Definition~\ref{defi:saturated}.
\item[(iii)] $A$ is fibrant in the Morita model structure of Theorem~\ref{thm:morita_model}.
\item[(iv)] The unique $*$-functor $A\to \mathbf0$ has the right lifting property with respect to the set $S=\{\, R_0\colon \emptyset\to \mathbf0 \, , \, R_1\colon P(1)\to R(1) \,,\,  S_2\colon \F^2 \to S(2)\, \}$.
\end{enumerate}
\end{prop}
\begin{proof}
The equivalence (i)$\Leftrightarrow$(ii) was proved in Proposition \ref{prop:UP_saturated}(ii). Let us now show that (ii) $\Rightarrow$~(iii). By definition of the set of generating trivial cofibrations $J_\Mor$, the $\Cstar$-category $A$ is fibrant if and only if every (solid) diagram 
\[
\xymatrix{ 
& P(n)\ar[r]^-F \ar[d]_{R_n} & A \\
& R(n) \ar@{..>}[ur]_{\overline{F}} &
}
\]
admits a $\ast$-functor $\overline{F}$ making the triangle commute. 
For $n=0$, the existence of $\overline{F}$ is equivalent to the existence in~$A$ of a zero object.
For $n\geq1$, the $\ast$-functor~$F$ corresponds to the choice of~$n$ objects $F(o_1), \ldots, F(o_n)$ of $A$ plus the choice of an idempotent self-adjoint $n\times n$ matrix of arrows $F(p_{ij}), 0 \leq i,j \leq n$, between them; see Lemma~\ref{lem:char}. 
A $\ast$-functor $\overline{F}$ as above corresponds to the choice of a range object for the projection matrix $[F(p_{ij})]$. Since by hypothesis $A$ is saturated, such a range object exists and hence $\overline{F}$ also exists.
Indeed, by saturation we can choose in $A$ a direct sum $x$ of $F(o_1),\ldots,F(o_n)$, with isometries $v_i\colon F(o_i)\to x$, and also a range object~$r$, with isometry $v:r\to x$, for the projection $p\in A(x,x)$ uniquely determined by $[Fp_{ij}]$. 
Then we may set $\overline{F}(r(n)):=r$ and $\overline{F}(s_i):=v^*v_i$, and by construction $\overline{F}R_n=F$.
This proves the implication (ii) $\Rightarrow$~(iii).

If (iii) holds, then by definition the $*$-functor $A\to \mathbf0$ has the right lifting property with respect to $R_0$, $R_1$ and also $S_2$ thanks to Lemma \ref{lemma:RLP_SvsR}. Thus (iii) $\Rightarrow$~(iv).

 So it remains only to prove the implication (iv) $\Rightarrow$~(ii).
 Property~(iv) means that all (solid) diagrams of the following three shapes
\[
\xymatrix{ 
\emptyset \ar[r] \ar[d] & A && \F^2 \ar[r]^-F \ar[d]_{S_2} & A && P(1)\ar[r]^-G \ar[d]_{R_1} & A \\
\mathbf 0 \ar@{..>}[ur] & && S(2) \ar@{..>}[ur]_{\overline{F}} &&& R(1) \ar@{..>}[ur]_{\overline{G}} &
}
\]
admit liftings as indicated by the dotted maps. As before, the existence of the left-most lifting simply means that $A$ has a zero object. The existence of a lifting $\overline{F}$ means that $A$ admits a direct sum for any two objects. An easy recursive argument then shows that $A$ admits all finite sums. Finally, the $*$-functor $G$ corresponds to the choice of a projection $p\in A(x,x)$. Hence, a lifting $\overline{G}$ exists if and only if $p$ splits in~$A$. Making use of Corollary~\ref{cor:idem} as usual, we then conclude that~$A$ is idempotent complete. In sum, $A$ is additive and idempotent complete, and so by definition saturated. 
\end{proof}

\begin{cor}
\label{cor:fibrant_replacement}
The $*$-functor $\iota_A\colon A\to A^\natural_\oplus$ provides a canonical choice for a fibrant replacement functor in the Morita model category $\cM_\Mor$.
\qed
\end{cor}

\subsection{Morphisms in the Morita homotopy category}
We are now well equipped to provide a complete description of the morphisms in $\Ho(\cM_\Mor)$, and characterize those which are invertible. 

\begin{prop}
\label{prop:morphisms_Morita}
Consider an arbitrary morphism $\varphi\colon A\to B$ in the Morita homotopy category $\Ho(\cM_\Mor)$. Then the following holds:
\begin{itemize}
\item[(i)]
The morphism $\varphi$ is represented by a fraction of the form
\[
\xymatrix@1{A \ar[r]^-F & B^\natural_\oplus & B \ar[l]^-{\sim}_-{\iota_B} }
 \]
 for some $*$-functor~$F$. 
\item[(ii)]
The morphism $\varphi$ is invertible if and only if the essentially unique extension $\widetilde{F}\colon A^\natural_\oplus \to B^\natural_\oplus$ of~$F$ along $\iota_A$ provided by Prop.~\ref{prop:UP_saturated} is a unitary equivalence.  
\item[(iii)]
Two $*$-functors $F_1,F_2\colon A\to B^\natural_\oplus$ represent the same morphism $\varphi\colon A\to B$, as in part~(i), if and only if they are (unitarily) isomorphic in $\Cstar(A,B^\natural_\oplus)$.
\end{itemize}
\end{prop}

\begin{cor}\label{cor:new}
Given $\Cstar$-categories $A$ and $B$ there is a canonical bijection
$$ \Hom_{\Ho(\cM_\Mor)}(A,B)\simeq \mathrm{ob}(\Cstar(A,B^\natural_\oplus))/_{\!\simeq}\,,$$
where the equivalence relation $\simeq$ is unitary isomorphism.
\qed
\end{cor}

\begin{proof} In the homotopy category $\Ho(\mathcal M)=\mathcal M[\Weq^{-1}]$ of any model category~$\mathcal M$, an arbitrary morphism $X\to Y$ can always be represented by a zig-zag of the form $X\stackrel{\sim}{\leftarrow} X'\to Y' \stackrel{\sim}{\leftarrow} Y$, where $X'\stackrel{\sim}{\to} X$ is some cofibrant replacement of~$X$ and $Y \stackrel{\sim}{\to} Y'$ some fibrant replacement of~$Y$.
Hence item (i) follows from Corollary \ref{cor:fibrant_replacement} and from the fact that every $\Cstar$-category is Morita cofibrant.  

In what concerns item (ii) consider the following commutative diagrams:
\[
\xymatrix{A \ar[d]_{\iota}^\sim \ar[dr]^F & \\ A^\natural_\oplus \ar[r]_-{\widetilde{F}} & B^\natural_\oplus}
\quad\quad\quad
\xymatrix{
A^\natural_\oplus \ar[r]^-{\widetilde{F}} \ar[d]_{\iota}^\simeq &
 B^\natural_\oplus \ar[d]_\simeq^{\iota} \\
(A^\natural_\oplus)^\natural_\oplus \ar[r]_-{(\widetilde F)^\natural_\oplus} & (B^\natural_\oplus)^\natural_\oplus\,.
}
\]
The commutativity of the above triangle implies that $F$ is a Morita equivalence if and only if $\widetilde{F}$ is one. By applying $(-)^\natural_\oplus$ to $\widetilde F$ we obtain the commutative square on the right-hand side. Since its vertical arrows are unitary equivalences by Proposition \ref{prop:UP_saturated}(ii), we observe that $\widetilde{F}$ is a Morita equivalence if and only if it is a unitary equivalence. The conclusion now follows by combining these two observations. 

The proof of item (iii) makes use of some facts about cylinders, left homotopy, etc., for which we refer to \cite{dwyer-spalinski}*{\S4}. 
Consider in $\Cstarcatun$ the factorization 
\begin{equation}\label{eq:cyl}
\xymatrix{
A\sqcup A \ar[dr]_{(J_1,J_2)} \ar[rr]^-{(\id,\id)} && A \\
& A\otimes_\mathrm{max} \mathbf I \ar[ur]_Q^\simeq &
}
\end{equation}
of the fold map $\nabla=(\id_A,\id_A)$ of a $\Cstar$-category~$A$. Here $A\otimes_\mathrm{max}\mathbf I$ is the maximal tensor product (see \cite{ivo:unitary}*{\S2}) of~$A$ with the ``interval'' $\Cstar$-category~$\mathbf I$ (see \cite{ivo:unitary}*{Def.\,3.4}), $J_i$ sends $x\in \obj(A)$ to $(x,i)\in \obj(A\otimes_\mathrm{max} \mathbf I)$ ($i=1,2$), and $Q$ is the evident projection. 
This diagram is a cylinder object in the unitary model category (\cf\ \cite{ivo:unitary}*{Lemma 3.14}) because $Q$ is a unitary equivalence and $(J_1,J_2)$ a unitary cofibration. 
Since $\Weq_\uni\subset \Weq_\Mor$ and $\Cof_\uni =\Cof_\Mor$, we see that \eqref{eq:cyl} is also a (at least ``good'') cylinder object for $\cM_\Mor$.
Now, two morphisms $F_1,F_2\colon A\to B^\natural_\oplus$ represent the same map~$\varphi$ in $\Ho(\cM_\Mor)$, as in item~(i), if and only if they become equal in $\Ho(\cM_\Mor)$ (because $\iota_B$ becomes an isomorphism). Since $A$ is Morita cofibrant and $B^\natural_\oplus$ is Morita fibrant, 
this is the case if and only if there exists a left homotopy $F_1\sim^\ell F_2$ with respect to any choice of cylinder over~$A$.
Choosing \eqref{eq:cyl}, this is the case if and only if there exists a $*$-functor $H\colon A\otimes_\mathrm{max}\mathbf I\to B^\natural_\oplus$ such that $H\circ J_i=F_i$ for $i=1,2$. Since the data of such an $H$ is equivalent to that of a natural unitary isomorphism $F_1\stackrel{\sim}{\to}F_2$, the claim is proved. 
\end{proof}

\begin{remark}
Note that the argument used in the proof of item (ii) also proves the following more general fact: a $*$-functor between two saturated $\Cstar$-categories is a Morita equivalence if and only if it is a unitary equivalence. 
This can also be obtained as a formal consequence of Proposition \ref{prop:bousfield} below.
\end{remark}

It is now clear how to compose maps in the homotopy category.

\begin{cor}
\label{cor:comp}
If two composable morphisms $\varphi\colon A\to B$ and $\psi\colon B\to C$ of $\Ho(\cM_\Mor)$ are represented, as in Proposition \ref{prop:morphisms_Morita}(i), by the $*$-functors $F\colon A\to B^\natural_\oplus$ and $G\colon B\to C^\natural_\oplus$, their composition $\psi\circ \phi$ is represented by the $*$-functor $\widetilde{G}\circ F$.
\qed
\end{cor}

\subsection{Morita equivalence of unital $\Cstar$-algebras}
\label{subsec:morita_agreement}
In the theory of $\Cstar$-algebras one encounters more than one notion of Morita equivalence. The most commonly used is the one introduced by Rieffel called \emph{Morita-Rieffel equivalence}, \emph{strong Morita equivalence}, or sometimes just \emph{Morita equivalence}. It is based on the highly structured notion of imprimitivity bimodules~\cites{rieffel:induced, rieffel:morita, brown-green-rieffel}; see \cite{raeburn-williams} for a textbook treatment and  \cite{blecher:new} for an alternative viewpoint. In the case of unital $\Cstar$-algebras -- the ones also appearing as objects in $\Cstarcatun$ --  things simplify considerably because in this case Morita-Rieffel equivalence actually agrees with the usual notion of Morita equivalence of rings. Without going into too much detail, we now connect these classical ideas with the constructions introduced in the present article.

\begin{prop}
\label{prop:Morita_agreement}
For unital  $\Cstar$-algebras $A$ and $B$, the following are equivalent:
\begin{itemize}
\item[(i)] $A$ and $B$ are isomorphic objects in the Morita homotopy category~$\Ho(\cM_\Mor)$.
\item[(ii)] $A$ and $B$ are strongly Morita equivalent, in the usual sense of $\Cstar$-algebras that there exists an imprimitivity bimodule between them. 
\item[(iii)] Considered as rings,  $A$ and $B$ are Morita equivalent in the usual sense.
\end{itemize}
\end{prop}

\begin{proof}
The equivalence (ii)$\Leftrightarrow$(iii) is well-known and due to Beer; see \cite{beer:morita}*{\S1.8}.
Let us now sketch a proof of the equivalence (i)$\Leftrightarrow$(ii) (the operator-algebraist will have no problem in filling in the missing details). To this end we use the fact that for every unital $\Cstar$-algebra~$D$, there is a unitary equivalence of $\Cstar$-categories 
\begin{equation} \label{eq:identif_proj}
D_\oplus^\natural \stackrel{\sim}{\to} \Hilb(D)_{\proj}
\end{equation}
obtained by extending the Yoneda embedding $D\stackrel{\sim}{\to} \mathcal L(D) \subset \Hilb(D)$ by Proposition \ref{prop:UP_saturated}(i); here $\Hilb(D)$ denotes the $\Cstar$-category of (right) Hilbert $D$-modules and adjointable operators between them (see \eg\ \cite{lance:toolkit}), and $\Hilb(D)_{\proj}$ denotes its full sub-$\Cstar$-category of those Hilbert modules which are finitely generated projective over~$D$.

Now let $A$  and~$B$ be two unital $\Cstar$-algebras. Recall that an imprimitivity bimodule is the same as an invertible correspondence, and for unital $\Cstar$-algebras this is the same as a unital $*$-homomorphism $\varphi\colon A\to \mathcal L(E)$ into the $\Cstar$-algebra of adjointable operators on a finitely generated projective right Hilbert $B$-module~$E$, with the property of  inducing a unitary equivalence $(-)\otimes_{\varphi} E \colon \Hilb(A)_\proj\stackrel{\sim}{\to} \Hilb(B)_\proj$. 
Under \eqref{eq:identif_proj}, the $*$-functor $(-)\otimes_{\varphi} E$ corresponds to a unitary equivalence 
$F\colon A_\oplus^\natural \stackrel{\sim}{\to} B^\natural_\oplus $. By Proposition \ref{prop:UP_saturated}(i), the $*$-functor $F$ is determined, up to a unitary isomorphism, by its restriction $F\circ \iota_A\colon A\to B^\natural_\oplus$. 
Therefore we see that $A$ and $B$ are Morita-Rieffel equivalent if and only if there exists a $*$-functor $\varphi\colon A\to B^\natural_\oplus$ which extends to a unitary equivalence $A_\oplus^\natural \stackrel{\sim}{\to} B^\natural_\oplus$. By Proposition \ref{prop:morphisms_Morita}, this is equivalent to $A$ and $B$ becoming isomorphic in the Morita homotopy category and so the proof is finished.
\end{proof}

\section{Picard groups}\label{sec:Pic}
The Picard group $\Pic(A)$ of a $\Cstar$-algebra $A$ encodes a lot of interesting information. Recall for instance from \cite{brown-green-rieffel}*{page\,357} that when $X$ is a compact Hausdorff space and $A=C(X)$ is the algebra of continuous functions on~$X$, we have the following identification
\[
\Pic(A) \simeq \mathrm{Homeo}(X)\ltimes \Pic(X)\,,
\]
where the right-hand-side is the semi-direct product of the homeomorphism group of $X$ with its Picard group of  line bundles. 

\subsection*{Proof of Proposition~\ref{prop:Pic}}
Recall from \cite{brown-green-rieffel}*{\S3} that $\Aut(A)$ is the group of (unitary) isomorphism classes of imprimitivity bimodules ${}_AM_A$, with group operation given by the Rieffel tensor product. 
As in the proof of Proposition \ref{prop:Morita_agreement}, we see that the unitary equivalence \eqref{eq:identif_proj} induces a bijection between $\Aut(A)$ and the set of (unitary) isomorphism classes of $*$-functors $F\colon A\to A^\natural_\oplus$ such that the extension $F\colon A^\natural_\oplus\to A^\natural_\oplus$ is a unitary equivalence. By Proposition \ref{prop:morphisms_Morita} this set can then be canonically identified with $\Aut_{\Ho(\cM_\Mor)}(A)$.
In order to complete the proof, it remains then to verify that the Rieffel tensor product of imprimitivity bimodules corresponds (up to isomorphism) to the composition in $\Ho(\cM_\Mor)$, as detailed in Corollary~\ref{cor:comp}. 
This is easy to verify because, when working up to isomorphism, we may forget the inner products on the imprimitivity bimodules; \cf\ the proof of \cite{beer:morita}*{Theorem \S1.8}. The problem reduces then to the  analogous algebraic statement, whose details we leave to the reader.
\section{Symmetric monoidal structure} \label{sec:symm}
In this section we show that the Morita model structure is nicely compatible with the maximal tensor product of $\Cstar$-categories $\otimes_{\mathrm{max}}$ introduced in \cite{ivo:unitary}*{\S2}. In order to simplify the exposition we will often simply write $\otimes$ instead of $\otimes_{\mathrm{max}}$.

\begin{lemma}\label{lemma:key-mon}
For every $C^\ast$-category $A$ and $n\geq0$, the induced $\ast$-functor
\begin{eqnarray*}
A\otimes R_n\colon A\otimes P(n) \too A \otimes R(n)
\end{eqnarray*}
is a Morita equivalence.
\end{lemma}

\begin{proof}
Recall that the $\ast$-functor $R_n$ is fully faithful. Hence, by construction, $A\otimes R_n$ is also fully faithful. By Proposition~\ref{prop:UP_saturated}(iii) we then conclude that $(A\otimes R_n)_\oplus^\natural$ is also fully faithful. It remains to show $(A\otimes R_n)^\natural_\oplus$ is unitarily essentially surjective. 
Recall that the objects of $A\otimes R(n)$ are pairs $x\otimes y:=(x,y)$ with $x\in\obj(A)$ and $y\in R(n)$. Thus an object in $(A \otimes R(n))_\oplus^\natural$ is given by 
\begin{equation} \label{eq:pair1}
z \, := \, \left((x_1\otimes y_1) \cdots (x_j \otimes y_j) \cdots (x_m \otimes y_m) \, , \, P\right)\,,
\end{equation}
where $P$ is a suitable $m\times m$ projection matrix, and each $y_j$ is an occurrence of~$r(n)$ or of~$o_i$ for some $i\in \{1,\ldots,n\}$. Now, if $j$ is such that $y_j=r(n)$ in \eqref{eq:pair1}, we can replace~$z$ by the following object of $(A\otimes R(n))_\oplus^\natural$: 
\begin{equation*} \label{eq:pair2}
z' \, :=\, \left[(x_1\otimes y_1) \cdots (x_j \otimes o_1) \cdots (x_j \otimes o_n) \cdots (x_m \otimes y_m) \, , \, UPV \right]\,,
\end{equation*}
where we have inserted $(x_j\otimes o_1)\cdots (x_j\otimes o_n)$ instead of $x_j\otimes y_j$, and where~$U$ denotes the block-matrix 
\[
U:=
\left[
\begin{array}{c|ccc|c}
\id_{j-1} &0&\cdots&0&0 \\ \hline
0& 1_{x_j} \otimes s_1 & \cdots & 1_{x_j} \otimes s_n &0 \\ \hline
0&0&\cdots&0& \id_{m-j}
\end{array}
\right]_{m \times (m+n-1)}
\]
and $V$ the block-matrix
\[
V := U^* =
\left[
\begin{array}{c|c|c}
\id_{j-1} &0&0 \\ \hline
0& 1_{x_j} \otimes s_1^*& 0\\ 
\vdots &\vdots & \vdots  \\ 
0& 1_{x_j} \otimes s_n^* & 0 \\ \hline
0&0& \id_{m-j} 
\end{array}
\right]_{(m+n-1) \times m}\,.
\]
Note that the objects $z$ and $z'$ are unitarily isomorphic via $U\colon z\stackrel{\sim}{\to} z'$ and~$V\colon z'\stackrel{\sim}{\to} z$. 
Repeating this procedure whenever necessary from $j=1$ to $j=m$, we find at the end an object~$z''$ in the image of $(A\otimes R_n)_\oplus^\natural$ and a unitary isomorphism $z\stackrel{\sim}{\to}z''$ in $(A\otimes R(n))_\oplus^\natural$.
This achieves the proof.
\end{proof}
\begin{prop}\label{prop:monoidal}
The category $\Cstarcatun$, endowed with the Morita model structure and with the closed symmetric monoidal structure induced by the maximal tensor product $\otimes_{\mathrm{max}}$, is a monoidal model category in the sense of \cite[Def.~4.2.18]{hovey:model}.
\end{prop}
\begin{proof}
Given two $\ast$-functors $F\colon A\to B$ and $F'\colon A'\to B'$, consider the following commutative diagram
$$
\xymatrix@!0 @R=3.5pc @C=6pc{
A\otimes A' \ar[dd]_-{F\otimes A'} \ar[rr]^-{A \otimes F'} && A\otimes B' \ar[dd]^-{F\otimes B'} \ar[dl]_-G\\
& (B \otimes A') \sqcup_{A \otimes A'} (A\otimes B') \ar@{..>}[dr]^{F\square F} & \\
B \otimes A' \ar[rr]_-{B \otimes F'} \ar[ur] && B \otimes B' \,,
}
$$
where $F\square F'$ is the $\ast$-functor induced by the pushout property. We need to verify the following two conditions:
\begin{itemize}
\item[(i)] the $\otimes$-unit $\Cstar$-category is cofibrant;
\item[(ii)] if $F, F' \in \Cof_\Mor$, then $F \square F' \in \Cof_\Mor$. Moreover, if $F$ or $F'$ is a Morita equivalence, then $F\square F'$ is also a Morita equivalence.
\end{itemize}
Since the unitary and the Morita model structure have the same cofibrations, condition~(i) is obvious (every object is cofibrant) and the first claim of condition~(ii) follows from \cite[Prop.~3.19]{ivo:unitary}. In what concerns the second claim of condition~(ii), since the Morita model structure is cofibrantly generated we can assume without loss of generality that $F \in J_\Mor$. By Lemma~\ref{lemma:key-mon}, the $\ast$-functors $F \otimes A'$ and $F\otimes B'$ in the above diagram are not only cofibrations but moreover Morita equivalences. 
Since trivial cofibrations are stable under pushouts, $G$ is also a Morita equivalence. Finally, by the $2$-out-of-$3$ property we conclude that $F\square F'$ is a Morita equivalence. This achieves the proof.
\end{proof}

\section{Simplicial structure} \label{sec:simplicial}
In this section we show that the Morita model structure is nicely compatible with the simplicial structure constructed in~\cite{ivo:unitary}. 
We recall from \cite{ivo:unitary}*{Def.~3.22} that the mapping complex of two $\Cstar$-categories $A$ and~$B$ is the simplicial set 
\[
\Map(A,B) := \nu \,\Cstar(A,B) \;, 
\]
where the functor $\nu:= N\circ uni$ associates to any $\Cstar$-category~$D$ the simplicial nerve,~$N$, of its subcategory $uni (D)$ of unitary isomorphisms. The remaining simplicial structure is given (for $K$ a simplicial set and $D$ a $\Cstar$-category) by the  the coaction $D^K := \Cstar(\pi K, D)$ and the action $D\otimes K := D\otimes_\mathrm{max} \pi (K)$, where $\pi:= \Cstarmax \circ \Pi$ is the functor associating to a simplicial set the maximal enveloping $\Cstar$-category of its (simplicial) fundamental groupoid; consult \emph{loc.\,cit.\ }for details.

\begin{prop}
\label{prop:simplicial}
The category $\Cstarcatun$, endowed with the Morita model structure and with the simplicial structure described above, is a simplicial model category in the sense of \cite[Def.~4.2.18]{hovey:model}.
\end{prop}
\begin{proof}
Given a $\ast$-functor $F\colon A \to B$ and a map $f\colon K \to L$ of simplicial sets, consider the following commutative diagram:
$$
\xymatrix@!0 @R=3.5pc @C=6pc{
A\otimes K \ar[dd]_-{F\otimes K} \ar[rr]^-{A \otimes f} && A\otimes L \ar[dd]^-{F\otimes L} \ar[dl]\\
& (B \otimes K) \sqcup_{A \otimes K} (A\otimes L) \ar@{.}[dr]^{F\square f} & \\
B \otimes K \ar[rr]_-{B \otimes f} \ar[ur] && B \otimes L \,.
}
$$
Since every $\Cstar$-category is cofibrant we need only to verify the following condition:
\begin{itemize}
\item[(i)] If $F \in \Cof_\Mor$ and $f \in \Cof_\sSet$, then $F\square f \in \Cof_\Mor$. Moreover, if $F$ is a Morita equivalence or $f$ is a weak equivalence, then $F\square f$ is also a Morita equivalence. 
\end{itemize}
Recall from \cite[Prop.~3.21]{ivo:unitary} the construction of the following Quillen adjunction
\begin{eqnarray}\label{eq:adjunction}
\xymatrix{
\Cstarcatun \ar@<1ex>[d]^\nu \\
\sSet \ar@<1ex>[u]^\pi\,,
}
\end{eqnarray}
where $\Cstarcatun$ is endowed with the unitary model structure. Since $\Cof_\uni=\Cof_\Mor$ and $\Cof_\uni \cap \Weq_\uni \subset \Cof_\Mor \cap \Weq_\Mor$, the above adjunction \eqref{eq:adjunction} is also a Quillen adjunction with respect to the Morita model structure. By definition, we have $A\otimes K = A \otimes \pi K$. Hence the above condition~(i) follows from condition~(ii) of the proof of Proposition~\ref{prop:monoidal}.
 \end{proof}

\section{Bousfield localization} \label{sec:bousfield}

In this section, making use of the mapping complex functor $\Map(-,-)\colon \Cstarcatun^\op\times \Cstarcatun\to \Cstarcatun$ (see \S\ref{sec:simplicial}), we characterize the Morita model structure as a left Bousfield localization; see \cite{hirschhorn}*{Def.~3.1.1}.
\begin{prop}
\label{prop:bousfield}
The Morita model structure is the left Bousfield localization of the unitary model structure with respect to the set 
$$S:=\{ \emptyset\to \mathbf0,\;  R_1\colon P(1) \to R(1) ,\; S_2\colon \bbF^2 \to S(2)\}\,.$$
\end{prop}
\begin{proof}
We need to prove the following statements:
\begin{itemize}
\item[(i)] A $\Cstar$-category $A$ is $S$-local (see \cite[Def.~3.1.4]{hirschhorn}) if and only if it is fibrant in the Morita model structure. 
\item[(ii)] A $\ast$-functor $F$ is an $S$-local equivalence (see \cite[Def.~3.1.4]{hirschhorn}) if and only if it is a Morita equivalence.
\end{itemize}
Let us begin with statement (i). By construction the unitary and the Morita model structures have the same cofibrations, and therefore the same trivial fibrations. As a consequence the simplicial cofibrant replacement functor $\Gamma^\ast$ (see \cite[\S17]{hirschhorn}) is the same in both cases. 
Since $S$ consists of Morita equivalences, this implies that if $A$ is fibrant in the Morita model structure, then $R_0$, $R_1$ and $S_2$ induce weak equivalences of simplicial sets
$$
\begin{array}{lcl}
\map_\uni(R_0,A)\colon \map_\uni(\mathbf0, A) &\stackrel{\sim}{\too}& \map_\uni(\emptyset, A) \\
 \map_\uni(R_1,A)\colon \map_\uni(R(1), A) &\stackrel{\sim}{\too}& \map_\uni(P(1), A) \\
\map_\uni(S_2,A)\colon \map_\uni(S(2), A) &\stackrel{\sim}{\too}& \map_\uni(\bbF^2, A)\,,
\end{array}
$$
where $\map_\uni(-,-)$ stands for the homotopy function complex in the unitary model structure; see \cite[Notation~17.4.2]{hirschhorn}. 
In other words, if $A$ is fibrant in the Morita model structure, then it is $S$-local. 
Let us now prove the converse. By Proposition~\ref{prop:fibrant}, $A$ is fibrant in the Morita model structure if and only if it has a zero object and all diagrams of shape
 \[
\xymatrix{ 
\F^2 \ar[r]^-F \ar[d]_{S_2} & A && P(1)\ar[r]^-G \ar[d]_{R_1} & A \\
S(2) \ar@{..>}[ur]_{\widetilde{F}} &&& R(1) \ar@{..>}[ur]_{\widetilde{G}} &
}
\]
admit liftings $\overline{F}$ and~$\overline{G}$. Since the unitary model is simplicial (see \cite{ivo:unitary}*{Thm.~3.23}) and every object is fibrant and cofibrant in the unitary model, we can choose $\map_\uni(-,-)$ to be the mapping complex $\Map(-,-)$ (again, consult \cite[\S17]{hirschhorn}).
Using this identification and the hypothesis that $\Map(R_0,A)$ is a weak equivalence, we observe that $\Map(\mathbf0,A)$ is weakly equivalent to a point $\Map(\emptyset,A) = \{\pt\}$. In particular it is nonempty, and so $\Cstar(\mathbf0,A)$ has an object, and hence $A$ has a zero object. Now, note that $F$ (in the above diagram) is a $0$-simplex of $\Map(\bbF^2,A)$. Since by hypothesis the induced map $\Map(S_2,A)$ of simplicial sets is a weak equivalence, there exists a $0$-simplex $\widetilde{F}$ of $\Map(S(2),A)$, \ie\ a $\ast$-functor $\widetilde{F}:S(2) \to A$, and a zig-zag of $1$-simplices in $\Map(\bbF^2,A)$ relating $F$ with $\widetilde{F}\circ S_2$. 
From the definition of $\Map(-,-)$ we observe that the above zig-zag of $1$-simplices reduces to a single $1$-simplex $H$ in $\Map(\bbF^2,A)$. This data corresponds precisely to the following diagram in $A$
$$
\xymatrix@!0 @R=1.5pc @C=6pc{
F(\bullet_1) \ar[r]^-{H(\bullet_1)}_\sim &\widetilde{F}(\bullet_1) \ar[dr]^{\widetilde{F}(v_1)}  & \\
 && \widetilde{F}(s(2)) \,,\\
F(\bullet_2) \ar[r]_-{H(\bullet_1)}^\sim & \widetilde{F}(\bullet_2) \ar[ur]_{\widetilde{F}(v_2)} & 
}
$$
satisfying the relations
\begin{eqnarray*}
\widetilde{F}(v_1) \circ \widetilde{F}(v_1^\ast) + \widetilde{F}(v_2)\circ \widetilde{F}(v_2^\ast) = 1_{\widetilde{F}(s(2))} &  & F(v_i^\ast)\circ F(v_j)= \delta_{ij}\,,
\end{eqnarray*}
where $H(\bullet_1)$ and $H(\bullet_2)$ are moreover unitary morphisms. 
It follows that we can define a $\ast$-functor $\overline{F}$ (making the above left-hand-side triangle commutative) by setting
\begin{equation*}
\bullet_i \mapsto F(\bullet_i) 
\quad\quad
 s(2) \mapsto \widetilde{F}(s(2)) 
 \quad\quad
 v_i \mapsto \widetilde{F}(v_i) \circ H(\bullet_i)\,.
\end{equation*}
Similarly, $G$ is a $0$-simplex of $\Map(P(1),A)$. Since $\Map(R_1,A)$ is a weak equivalence there exists a $0$-simplex $\widetilde{G}$ of $\Map(R(1),A)$, \ie\ a $\ast$-functor $\widetilde{G}\colon R(1) \to A$ and a $1$-simplex $H$ in $\Map(P(1),A)$ relating $G$ with $\widetilde{G}\circ R_1$. This data corresponds precisely to the following commutative diagram
$$
\xymatrix{
G(o_1) \ar[d]_{G(p_{11})} \ar[r]^{H(o_1)} & \widetilde{G}(o_1) \ar[d]^{\widetilde{G}(p_{11})} \ar[r]^{\widetilde{G}(s_1)} & \widetilde{G}(r(1)) \ar@{=}[d] \\
G(o_1) \ar[r]_{H(o_1)} & \widetilde{G}(o_1) \ar[r]_{\widetilde{G}(s_1)} & \widetilde{G}(r(1))\,,
}
$$
satisfying the relations
\begin{eqnarray*}
\widetilde{G}(s_1^\ast)  \widetilde{G}(s_1) = (\widetilde{G}(s_1^\ast)  \widetilde{G}(s_1))^2 && \widetilde{G}(s_1)  \widetilde{G}(s_1^\ast) = 1_{\widetilde{G}(r(1))}\,,
\end{eqnarray*}
where $H(o_1)$ is moreover a unitary morphism. It follows that we can define a $\ast$-functor $\overline{G}$ (making the above right-hand-side triangle commutative) by setting
\begin{equation*}
o_1 \mapsto G(o_1) \quad\quad
 r(1) \mapsto \widetilde{G}(r(1)) \quad\quad
  s_1 \mapsto \widetilde{G}(s_1)  H(o_1)\,.
\end{equation*}
This concludes the proof that a $\Cstar$-category $A$ is $S$-local if and only if it is fibrant in the Morita model structure. 

Let us now prove statement (ii). The $\ast$-functor $F$ is an $S$-local equivalence if and only if for every $S$-local $\Cstar$-category $A$ the induced map $\map_\uni(F,A)$ is a weak equivalence of simplicial sets. By part~(i) this is equivalent to the condition that $\map_\uni(F,A)$ is a weak equivalence for every Morita fibrant $C^\ast$-category~$A$. Since, for such~$A$, we have an isomomorphism $\map_\uni(-,A) \simeq \map_{\Mor}(-,A)$ between the homotopy function complexes for the unitary and the Morita model structures, we conclude that the last condition is equivalent to requiring that~$F$ be a Morita equivalence. The proof is then achieved.
\end{proof}

\begin{remark}
It is unclear to the authors if the unitary model structure is cellular, and thus if the (left) Bousfield localization machinery developed in~\cite{hirschhorn} can be used to obtain the Morita model category. This would provide another way of establishing Theorem~\ref{thm:morita_model}.

\end{remark}

\section{Semi-additivity}\label{sec:semi_add}
In this section we explore the rich structure of the Morita homotopy category $\Ho(\cM_\Mor)$. We start with some general results concerning semi-additive categories.

A category $\cC$ is called {\em semi-additive} if it has a zero object $0$ (\ie\ an object which is simultaneously initial and terminal), finite products, finite coproducts, and is such that the canonical map from the coproduct to the product
\[\big[{}^1_0 \, {}^0_1 \big]\colon x \sqcup y \stackrel{\sim}{\too} x \times y\]
is an isomorphism for every pair of objects~$x,y$. Since all other finite (co)products can be obtained recursively from these, it follows that the analogous canonical map is an isomorphism for every finite collection of objects. 
A {\em semi-additive functor} $F:\cC \to \cD$ is a functor which preserves the zero object, finite products, and finite coproducts. 


\begin{lemma}\label{lem:prod-star}
In any semi-additive category~$\cC$, each Hom-set $\Hom_\cC(x,y)$ is naturally endowed with an abelian monoid structure~$+$. Moreover, composition is bilinear with respect to~$+$. In other words, $\cC$ is an $\N$-category, \ie, is a category enriched over the symmetric monoidal category of abelian monoids; see \cite{kelly:enriched_book}.
\end{lemma}
\begin{proof}
Given maps $f,g \in \Hom_\cC(x,y)$, their sum $f+ g$ is defined as the composition
$$x \stackrel{\Delta}{\too} x \times x \stackrel{f\times g}{\too} y \times y  \stackrel{\big[{}^1_0 \, {}^0_1 \big]^{-1}}{\too} y \sqcup y \stackrel{\nabla}{\too} y\,,$$
where $\Delta=\big[ {}^1_1 \big]$ and $\nabla=\big[1\;1\big]$ are the diagonal and fold maps induced by the universal properties. The neutral element $0$ for $+$ is given by the unique zero map $x\to 0\to y$.
The fact that $+$ is commutative (with neutral element~$0$) and that the composition operation is bilinear with respect to $+$ are simple exercises that we leave for the reader. 
\end{proof}
\begin{remark}
Note that for any semi-additive functor $F\colon \cC \to \cD$ the induced maps $\Hom_{\cC}(x,y) \to \Hom_{\cD}(Fx,Fy)$ are homomorphisms of abelian monoids. In fact, the latter is an equivalent condition; cf.\ \S\ref{sec:K0}.
\end{remark}

\begin{thm}\label{thm:semi-additive}
The Morita homotopy category $\Ho(\cM_\Mor)$ is semi-additive. 
\end{thm}
\begin{proof}
The homotopy category $\Ho(\mathcal M)$ of a Quillen model category $\mathcal{M}$ is always endowed with arbitrary small products and coproducts. These are obtained as appropriate total derived functors of the product and coproduct functors on~$\mathcal M$; see \cite{hovey:model}*{Example 1.3.11} for details. Thus, in our case, the coproduct and product of $A_1$ and $A_2$ in $\Ho(\cM_\Mor)$ are given respectively by $A_1\sqcup A_2$ and $(A_1)^\natural_\oplus\times (A_2)^\natural_\oplus$, with $\sqcup$ and~$\times$ denoting as usual the coproduct and product in $\Cstarcatun$. Moreover $\Ho(\cM_\Mor)$ is pointed, \ie\ the unique map from the initial to the final object is invertible. Indeed the $*$-functor $\emptyset\to \mathbf0$ is the generating trivial cofibration~$R_0$. In order to show that $\Ho(\cM_\Mor)$ is semi-additive, it remains then to prove that the canonical comparison map
\[
\big[ {}^1_0 \, {}^0_1 \big] \colon A_1\sqcup A_2 \longrightarrow (A_1)^\natural_\oplus\times (A_2)^\natural_\oplus
\]
is invertible. Recall that the diagonal components of $\big[ {}^1_0 \, {}^0_1 \big]$ are the identity $*$-functors of $A_1$ and $A_2$, respectively, and that the off-diagonal components are the zero maps, which exist since $\Ho(\cM_\Mor)$ is pointed.
Concretely, $\big[ {}^1_0 \, {}^0_1 \big]$ is given by the $*$-functor sending $x\in \obj(A_1)$ to $(x,0)$ and  $y\in \obj(A_2)$ to $(0,y)$. This functor is clearly fully faithful since by construction there are only zero morphisms in $A_1\sqcup A_2$ between $A_1$ and $A_2$.
We claim that every object in $(A_1)^\natural_\oplus\times (A_2)^\natural_\oplus$ can be obtained from objects in the image of $\big[ {}^1_0 \, {}^0_1 \big]$ by taking finite direct sums and retracts.
Notice that every object $(x,y)$ in the image of $\big[ {}^1_0 \, {}^0_1 \big]$ is a direct sum $(x,0)\oplus(0,y)$.
Let 
\[
z = \left( (x_1\cdots x_n, p) \, ,\, (y_1\cdots y_m, q) \right)
\]
be an arbitrary object of $(A_1)^\natural_\oplus\times (A_2)^\natural_\oplus$, with $x_1,\ldots,x_n \in\obj( A_1)$ and $y_1,\ldots,y_n\in\obj( A_2)$.
We then have the following identifications
\begin{eqnarray*}
z &\simeq&  \big( ((x_1,0) \cdots (x_n,0) , p) \,,\, ( (0,y_1)\cdots (0,y_m), q) \big) \\
 &\simeq &  ((x_1,0) \cdots (x_n,0) , p) \oplus ( (0,y_1)\cdots (0,y_m), q)\,.
\end{eqnarray*}
This implies that $z$ is a range object for the projection $\big[ {}^p_0\, {}^0_q\big]$  on the direct sum
\[
(x_1,0)\oplus \cdots \oplus (x_n,0) \oplus (0,y_1) \oplus \cdots \oplus (0,y_m) \,,
\]
where all summands belong to the image of the comparison $*$-functor $\big[{}^1_0\,{}^0_1\big]$. This proves our claim and so we conclude, making use of Lemma \ref{lemma:recognition_Moreq}, that $\big[{}^1_0\,{}^0_1\big]$ is a Morita equivalence and therefore an isomorphism in $\Ho(\cM_\Mor)$. The proof is finished.
\end{proof}

By combining Theorem~\ref{thm:semi-additive} with Lemma~\ref{lem:prod-star} we conclude that $\Ho(\cM_\Mor)$ has an intrinsic abelian monoid structure on each Hom-set. It can be described as follows:

\begin{thm}\label{thm:map_sum}
Given any two $\Cstar$-categories $A$ and~$B$, the direct sum operation on $B^\natural_\oplus$ turns the homotopy function complex $\map_\Mor(A,B)$ of the Morita model category 
into a monoid in the category of simplicial sets. On the set of connected components, the canonical isomorphism
\begin{equation}\label{eq:components}
\pi_0(\map_\Mor(A,B)) \simeq \Hom_{\Ho(\cM_\Mor)}(A,B)
\end{equation}
identifies this operation with the abelian monoid structure of Lemma~\ref{lem:prod-star}.
\end{thm}

\begin{proof}
Recall from \cite{hirschhorn}*{\S17} that, since the Morita model is simplicial (Prop.~\ref{prop:simplicial}), its homotopy function complexes admit the following description:
\[
\map_\Mor(A,B) 
= \Map (A, B^\natural_\oplus) 
\stackrel{\textrm{def.}}{=} \nu \, \Cstar(A,B^\natural_\oplus)
\, .
\]
For every $\Cstar$-category $B$, the (canonical) direct sum operation determines a strict symmetric monoidal structure on~$B^\natural_\oplus$ with unit object~$0$. In particular we have a functor $\oplus\colon B^\natural_\oplus \times B^\natural_\oplus\to B^\natural_\oplus $, which is clearly a $*$-functor. The functoriality of $\Cstar(-,-)$ induces then a $*$-functor 
\begin{equation}\label{eq:sum_Hom}
\oplus\colon \Cstar(A,B^\natural_\oplus) \times \Cstar(A,B^\natural_\oplus) 
\simeq \Cstar(A,B^\natural_\oplus\times B^\natural_\oplus)
\longrightarrow \Cstar(A,B^\natural_\oplus) \,.
 \end{equation} 
In this way, the objectwise sum of functors turns $\Cstar(A,B^\natural_\oplus)$ into a strictly associative symmetric monoidal category structure  (note that the symmetry isomorphism $\big[{}^0_1\,{}^1_0\big]$ is \emph{not} the identity, unless both factors are zero). The unit object is the composite $*$-functor $\smash{0_{A,B}\colon A\to \mathbf0\stackrel{0}{\to} B^\natural_\oplus }$, \ie\ the constant $*$-functor $x\mapsto 0$. By applying the product-preserving functor $\nu\colon \Cstarcatun \to \Cstarcatun$ to \eqref{eq:sum_Hom} we obtain a map 
\begin{eqnarray*}
 \map_\Mor(A,B) \times \map_\Mor(A,B) \too  \map_\Mor(A,B)
\end{eqnarray*}
of simplicial sets, which turns $\map_\Mor(A,B)$ into an abelian monoid with unit
\begin{equation*}
\{\pt\} \to \map_\Mor(A,B) \quad \quad \pt \mapsto 0_{A,B}
\end{equation*}
in the category of simplicial sets.
By the functoriality of $\pi_0$, we obtain on the set of connected components $\pi_0(\map_\Mor(A,B))$ the structure of a monoid with neutral element~$0_{A,B}$. 
Notice that the natural symmetry unitary isomorphism  $\big[{}^0_1\,{}^1_0\big]$ implies that this monoid is abelian.  Let us now verify that this structure coincides with the one given by Lemma~\ref{lem:prod-star} under the isomorphism \eqref{eq:components}. In both cases the neutral element is the unique map $0_{A,B}\colon A\to B$ in $\Ho(\cM_\Mor)$ factoring through zero, which as noted above is represented by the $*$-functor $\smash{ A\to \mathbf0\stackrel{0}{\to} B^\natural_\oplus }$.
Given two maps $f,g\in \pi_0(\map_\Mor(A,B))$, they are represented by $*$-functors $F\colon A\to B^\natural_\oplus$ and $G\colon A\to B^\natural_\oplus$. Their sum, according to the simplicial monoid structure, is represented by the pointwise direct sum, namely by the functor $F\oplus G\colon A\to B^\natural_\oplus\times B^\natural_\oplus$ sending $x\in\obj(A)$ to $F(x)\oplus G(x)$. The ``internal'' construction of $f+g$ from Lemma~\ref{lem:prod-star} is given by the following upper horizontal composition in $\Ho(\cM_\Mor)$:
\[
\xymatrix{
{f+g\colon \quad A} \ar[r]^-{\Delta \circ\iota } &
 A^\natural_\oplus \times A^\natural_\oplus \ar[r]^-{\widetilde{F}\times \widetilde{G}} &
 B^\natural_\oplus \times B^\natural_\oplus  
  \ar@<.5ex>[dr]^>>>>>>{\Phi} 
  \ar@{-->}[r]^-{\big[ {}^1_0 \, {}^0_1 \big]^{-1}}_-\sim &
 B^\natural_\oplus \sqcup B^\natural_\oplus \ar[r]^-{\nabla} 
  \ar[d]^\iota_\sim & 
 B^\natural_\oplus \\
&&& 
 (B^\natural_\oplus \sqcup B^\natural_\oplus)^\natural_\oplus \ar[ur]_{\widetilde{\nabla}} 
  \ar@<.5ex>[ul]^-{\big[ {}^1_0 \, {}^0_1 \big]^\sim} & \,.
}
\]
Except for the inverse of $\big[ {}^1_0 \, {}^0_1 \big]$, all maps in the composition $f+g$ are actual $*$-functors. 
Note that, by saturation of their targets, both $*$-functors $\big[ {}^1_0 \, {}^0_1 \big]$ and $\nabla$ admit extensions along~$\iota$ to $*$-functors $\big[ {}^1_0 \, {}^0_1 \big]^\sim$ and $\widetilde{\nabla}$, as pictured. Moreover $\big[ {}^1_0 \, {}^0_1 \big]^\sim$ admits the following (quasi-)inverse:
\[
\Phi\colon 
\big( (x_1\cdots x_n, p) \, , \, (y_1\cdots y_m,q) \big)
\mapsto 
\big( ( x_1^1\cdots x_n^1)\oplus (y_1^2\cdots y_m^2) \, , \,  p\oplus q \big)\,,
\]
where the upper indices $z^1$ and $z^2$ refer to whether we are considering an object $z\in \obj(B)$ as belonging to the first or to the second copy of $B$ inside $(B^\natural_\oplus \sqcup B^\natural_\oplus)^\natural_\oplus$.
Thus, in the above diagram, $f+g$ is also represented by the lower composition of $*$-functors from $A$ to~$B^\natural_\oplus$, which is immediately seen to yield $F\oplus G$ once again (up to a unitary isomorphism of functors). In conclusion, the two sum operations coincide at each Hom-set of the homotopy category. This achieves the proof.
\end{proof}

\subsection{Tensor product} \label{subsec:tensor}

It is an immediate consequence of Proposition~\ref{prop:monoidal} that the closed symmetric monoidal structure on $\Cstarcatun$, given by the maximal tensor product $\otimes_\mathrm{max}$ and the internal Hom functor $\Cstar(-,-)$, can be derived, thus inducing a closed monoidal structure on the Morita homotopy category $\Ho(\cM_\Mor)$. Using once again the fact that every object is cofibrant and that $(-)^\natural_\oplus$ is a fibrant replacement functor, we obtain the following result:
\begin{cor} \label{cor:tensor}
\label{cor:Ho_tensor} 
The Morita homotopy category $\Ho(\cM_\Mor)$ carries the structure of a closed symmetric monoidal category, with tensor given by the maximal tensor product $A\otimes_\mathrm{max} B$, internal Hom's given by
\[
\underline{\Hom}(A,B)= \Cstar(A, B^\natural_\oplus) \simeq \Cstar(A^\natural_\oplus, B^\natural_\oplus) \,,
\]
and tensor unit given by the $\Cstar$-algebra~$\F$.
\qed
\end{cor}

\section{Towards $K$-theory}\label{sec:K0}
In this section we start to investigate the relation between the Morita homotopy category $\Ho(\cM_\Mor)$ and  $K$-theory. 

\subsection{Group completion}
Recall that a semi-additive category $\mathcal C$ (see \S\ref{sec:semi_add}) is \emph{additive} if each Hom-set is an abelian group, \ie\ if every map has an additive inverse. Equivalently (by \cite{maclane}*{Thm.\,VIII.2.2}), $\mathcal C$ is additive if it is enriched over abelian groups, has a zero object~$0$, and a \emph{biproduct} of any two objects $x,y$, \ie\ a diagram of shape
\[ 
\xymatrix{ 
x \ar@<.5ex>[r]^-{i_x} & x\oplus y \ar@<.5ex>[l]^-{p_x} \ar@<-.5ex>[r]_-{p_y} & y \ar@<-.5ex>[l]_-{i_y}
}
\]  
verifying the relations:
\begin{equation*}
p_xi_x=1_x 
\quad\quad
p_yi_y=1_y
\quad\quad
 i_xp_x + i_yp_y = 1_{x\oplus y}\,.
\end{equation*}
Note that each biproduct gives rise to a product $(x\oplus y, p_x,p_y)$ and to a coproduct $(x\oplus y, i_x,i_y)$. Conversely one can easily obtain a biproduct out of a product-coproduct. This shows us that when $\cC$ is additive the addition of maps coincides with the abelian monoid structure of Lemma \ref{lem:prod-star}; see \cite{maclane}*{\S VIII.2 Prop.\,3, Ex.\,4}. 
In particular, we obtain:
\begin{cor}
\label{cor:unique_sum}
The abelian group enrichment of any additive category is uniquely and intrinsically determined.
\qed
\end{cor}

In fact, the proofs in \emph{loc.\,cit.\ }never use the existence of additive inverses of maps, hence the above observations, including Corollary~\ref{cor:unique_sum}, hold equally well for all semi-additive categories.


Recall that a functor between (semi-)additive categories is called {\em additive} if it preserves the additive structure of the Hom-sets or, equivalently, if it preserves the zero object and biproducts.
\begin{notation}\label{not:group-completion}
Given a semi-additive category $\cC$, let us denote by $\cC^{-1}$ the additive category obtained from $\cC$ by applying the group completion construction to every abelian monoid of morphisms. Note that $\cC$ and $\cC^{-1}$ have the same objects. The resulting canonical additive functor will be denoted by $\alpha: \cC \to \cC^{-1}$.
\end{notation}
\begin{prop}\label{prop:universal}
Let $\cC$ be a semi-additive category. Then, for any additive category~$\cD$ the canonical functor $\alpha\colon \cC \to \cC^{-1}$ induces an isomorphism of categories
$$ \alpha^\ast\colon \Fun_\add(\cC^{-1},\cD) \stackrel{\sim}{\too} \Fun_\add(\cC,\cD)\,,$$
where $\Fun_\add$ stands for the category of additive functors.
\end{prop}
\begin{proof}
Let $F\colon \cC \to \cD$ be an additive functor. For every ordered pair $(x,y)$ of objects in~$\cC$ the induced map $\Hom_\cC(x,y)\to \Hom_\cD(Fx,Fy)$ is a homomorphism. Hence, since $\cD$ is additive, this homomorphism has a unique $\Z$-linear extension 
\begin{equation}\label{eq:maps}
\Hom_{\cC^{-1}}(x,y)\stackrel{\mathrm{def.}}{=}(\Hom_{\cC}(x,y))^{-1}\too \Hom_{\cD}(Fx,Fy)
\end{equation}
to the group completion. The maps \eqref{eq:maps} assemble into an additive functor $\overline{F}\colon \cC^{-1}\to \cD$ such that $\overline{F}\circ \alpha=F$. 
One verifies that the assignment $F\mapsto \overline{F}$ gives rise to a functor which is strictly inverse to~$\alpha^*$. This achieves the proof.
\end{proof}
Recall from Theorem~\ref{thm:semi-additive} that the category $\Ho(\cM_\Mor)$ is semi-additive. Hence by applying to it the group completion construction described at Notation~\ref{not:group-completion} we obtain an additive category $\Ho(\cM_\Mor)^{-1}$. Consider the following composition
\begin{equation}\label{eq:univ-functor}
\cU\colon \Cstarcatun \stackrel{q}{\too} \Ho(\cM_\Mor) \stackrel{\alpha}{\too} \Ho(\cM_\Mor)^{-1}\,.
\end{equation}
\begin{thm}\label{thm:characterization}
The above functor \eqref{eq:univ-functor} sends Morita equivalences to isomorphisms, preserves all finite products (including the final object ${\bf 0}$), and is universal with respect to these properties, \ie\ given any additive category $\mathcal{A}$ we have an induced isomorphism of categories
$$ \cU^\ast\colon \Fun_\add(\Ho(\cM_\Mor)^{-1},\mathcal{A}) \stackrel{\sim}{\too} \Fun_{\Mor, \times}(\Cstarcatun,\mathcal{A})\,,$$
where the right-hand-side denotes the category of functors which invert Morita equivalences and preserve all finite products.
\end{thm}
\begin{proof}
By Proposition~\ref{prop:universal} it suffices to show that the quotient functor~$q$ induces an equivalence of categories
\begin{equation}\label{eq:induced}
q^\ast\colon \Fun_\add(\Ho(\cM_\Mor), \mathcal A) \stackrel{\sim}{\too} \Fun_{\Mor, \times}(\Cstarcatun, \mathcal A)\,.
\end{equation}
The proof will consist in constructing a (strict) inverse to $q^\ast$. The quotient functor is the localization of $\Cstarcatun$ with respect to the class of Morita equivalences. Hence by the universal property of localization we conclude that every functor $F$ belonging to $\Fun_{\Mor, \times}(\Cstarcatun, \mathcal A)$ admits a unique factorization $F\colon \Cstarcatun \stackrel{q}{\to} \Ho(\cM_\Mor) \stackrel{\overline{F}}{\to} \mathcal A$. 
Since $F$ preserves products and inverts Morita equivalences, we obtain the following isomorphism
\[
\xymatrix{
F(A^\natural_\oplus \times B^\natural_\oplus) \ar[r]_-{\sim} & 
  F(A^\natural_\oplus) \times F(B^\natural_\oplus) \ar[rr]_-\sim^-{(F\iota \times F\iota)^{-1}} &&
   F(A) \times F(B)
} 
\]
between objects of~$\mathcal A$.
As explained in the proof of Theorem~\ref{thm:semi-additive}, $A^\natural_\oplus \times B^\natural_\oplus$ is the product of $A$ and $B$ in $\Ho(\cM_\Mor)$. As a consequence we observe that $\overline F$ preserves products and thus that it is an additive functor.
Similarly, $\overline{F}$ preserves the final object $\mathbf0\stackrel{\sim}{\to}\mathbf0^\natural_\oplus$. Hence, since $\mathcal A$ is additive, we conclude that $\overline{F}$ is additive. Finally, a straightforward verification shows that the assignment $F \mapsto \overline{F}$ gives rise to a functor which is quasi-inverse to \eqref{eq:induced}.
\end{proof}

\subsection{Grothendieck group}

\begin{thm}\label{thm:co-representability}
For every unital $\Cstar$-algebra~$A$ there is a canonical isomorphism
\begin{equation}\label{eq:isom-Ko}
\Hom_{\Ho(\cM_\Mor)^{-1}}(\bbF,A) \simeq K_0(A)
\end{equation}
of abelian groups, where $K_0(A)$ denotes the classical Grothendieck group of~$A$.
\end{thm}

\begin{remark}
It is well-known that it does not matter whether we consider~$A$ as a unital $\Cstar$-algebra -- and then compute its Grothendieck group in terms of projections or Hilbert modules -- or as an ordinary algebra -- and then compute its Grothendieck group via idempotents or finitely generated projective modules. In the present paper, we have already seen an incarnation of this fact in Remark~\ref{remark:idemp_agreement}. (This of course ceases to be true for the higher $K$-theory groups).
\end{remark}

\begin{proof}
For the proof it will be convenient to use the customary picture of $K_0(A)$ in terms of projections; see \eg\ \cite{blackadar:Kth_op_alg}*{\S4}. Thus $K_0(A)$ is the group completion of the abelian monoid $(\mathcal V(A),\oplus,0)$ of projections in the nonunital $*$-algebra $M_\infty(A)=\bigcup_{n\geq1}M_n(A)$ with respect to the orthogonal sum $p\oplus q:= \left[ {}^p_0\, {}^0_q\right]$ and modulo unitary equivalence of projections, where $p$ and~$q$ are unitarily equivalent if and only if there exists a unitary matrix $u\in M_\infty(A)$ such that $u p u^*=q$ (here we are considering $M_n(A)$ as a subalgebra of $M_{n+1}(A)$ via the corner inclusion $a\mapsto \left[ {}^a_0\, {}^0_0\right]$).
It is now clear that every element $[p]$ of $\mathcal V(A)$ corresponds to a unique (unitary) isomorphism class of objects in $A^\natural_\oplus$, and therefore, by Proposition~\ref{prop:morphisms_Morita}, to a unique element of $\Hom_{\Ho(\cM_\Mor)}(\F,A)$. Hence we obtain a bijection
\begin{equation}\label{eq:bij}
\Hom_{\Ho(\cM_\Mor)}(\mathbb{F},A)\simeq \mathcal V(A)\,,
\end{equation}
which moreover identifies the sum operations~$\oplus$ and the zero elements~$0$ of both sides.  Thus \eqref{eq:bij} is an isomorphism of abelian monoids. By group completion we obtain then the required isomorphism \eqref{eq:isom-Ko}, which moreover is  natural with respect to unital $*$-homomorphisms (\ie\ $*$-functors) $A\to B$.
\end{proof}

\begin{remark}
\label{remark:Ktheory}
Theorem \ref{thm:co-representability} leads us naturally to define 
\begin{equation}\label{eq:K0def}
K_0(A):= \Hom_{\Ho(\cM_\Mor)^{-1}}(\F,A)
\end{equation}
for every small unital $\Cstar$-category $A$. One can already find in the literature a few definitions of the $K_0$-group of a $\Cstar$-category. Hence let us briefly compare our definition with some of the extant ones.
\begin{enumerate}
\item[(i)]
By forgetting structure, every (additive) $\Cstar$-category is a Banach category and so one can study its topological $K$-theory in the sense of Karoubi \cites{karoubi:clifford,karoubi:K} for which $K_0(A)$ is simply the Grothendieck group of~$A$ seen as an additive category. 
Thus the group \eqref{eq:K0def} agrees with Karoubi's $K_0$-group of~$A^\natural_\oplus$.
 In fact, for the purposes of Karoubi's $K$-theory, Banach categories are naturally assumed to be additive and idempotent complete.  This provided the original motivation for defining the idempotent complete (``pseudo-abelian'') hull of an additive category.
\item[(ii)]
Davis and L\"uck~\cite{davis-lueck} have introduced a nonconnective topological $K$-theory spectrum functorially associated to a $\Cstar$-category~$A$. 
Unfortunately, as pointed out in~\cite{joachim:KC*}, this spectrum is not well-defined because of a common mistake explained by Thomason in \cite{thomason:beware}. Nonetheless, their $K_0$-group still makes sense and is defined to be $\pi_0 (B (iso \, A^\natural_\oplus)\textrm{\textasciicircum} )$, \ie\ the path components of the classifying space of the Quillen-Grothendieck completion of the symmetric monoidal groupoid of isomorphisms in $A^\natural_\oplus$; see \cite{davis-lueck}*{page\,217}. This  is easily seen to be the Grothendieck group of $(A^\natural_\oplus, \oplus, 0)$ and thus coincides with \eqref{eq:K0def}. 

\item[(iii)]
In order to circumvent the Davis-L\"uck construction, Joachim \cite{joachim:KC*} and Mitchener \cite{mitchener:symm} have defined topological $K$-theory spectra (and therefore $K$-theory groups) for a $\Cstar$-category~$A$ by means of Hilbert modules over~$A$ with additional data. Both definitions are related to $KK$-theory, which gives them a strong analytical flavor. It is unclear to the authors, even for $K_0$, whether they agree with each other and with  
\eqref{eq:K0def}, although we expect this to be the case at least for a gentle enough $\Cstar$-category $A$.

\item [(iv)]
Kandelaki \cite{kandelaki:KK_K}  has adapted Karoubi's $K$-theory to $\Cstar$-categories, and used it to compute Kasparov groups $KK_n(A,B)$ as the $K$-groups of suitable saturated $\Cstar$-categories $\mathrm{Rep}(A,B)$. Here $K_0$ of a $\Cstar$-category~$A$ (with direct sums) is defined to be the Grothendieck group of the abelian monoid of unitary isomorphism classes of its objects. 
Therefore here too our group \eqref{eq:K0def} agrees with $K_0(A^\natural_\oplus)$.
\end{enumerate}
\end{remark}


\subsection{Pairings and multiplication}

We begin this subsection by proving a general property of the group completion $(-)^{-1}$ construction of semi-additive categories (see Notation \ref{not:group-completion}).

\begin{lemma}
\label{lemma:add_tensor}
Let $(\cC, \otimes , \unit)$ be a semi-additive category $\cC$ endowed with a symmetric monoidal structure whose tensor product is additive in each variable. Then, the associated additive category $\cC^{-1}$ is endowed with a unique symmetric monoidal structure which is again additive in both variables. Moreover, the canonical functor $\alpha\colon \cC\to \cC^{-1}$ is strict symmetric monoidal.
\end{lemma}

\begin{proof}
Since $\cC^{-1}$ is additive, Proposition \ref{prop:universal} applied to the semi-additive category $\cC\times\cC$ implies that the functor $\alpha\circ \otimes \colon \cC\to \cC\to \cC^{-1}$ admits a unique extension along $\alpha\colon \cC\times \cC\to (\cC\times \cC)^{-1}$. By composing it with the canonical isomorphism $\cC^{-1}\times\cC^{-1}\simeq(\cC \times \cC)^{-1}$ we obtain a tensor product on $\cC^{-1}$:
\[
\xymatrix@1{
& \cC\times \cC \ar[dl]_{\alpha\times\alpha} \ar[r]^-\otimes \ar[d]^\alpha & \cC \ar[d]^{\alpha}  \\
 \cC^{-1} \times \cC^{-1} \ar@/_3ex/[rr]_-\otimes \ar[r]^-\simeq & (\cC\times \cC)^{-1} \ar@{..>}[r]^-{\exists !} & \cC^{-1} \,.
}
\]
When combined with $\unit$ and with (the images in $\cC^{-1}$ of) the coherence isomorphisms of $(\cC,\otimes ,\unit)$, this data specifies a symmetric monoidal structure as required. 
\end{proof}
By combining Corollary \ref{cor:Ho_tensor} and Lemma \ref{lemma:add_tensor} we obtain on $\Ho(\cM_\Mor)^{-1}$ a well-defined symmetric monoidal structure~$\otimes$. Let $A$ and $B$ be two $\Cstar$-categories. Making use of this tensor product and of the natural isomorphisms \eqref{eq:isom-Ko} we obtain a well-defined external pairing
\begin{equation}\label{eq:pairing}
K_0(A) \otimes_\Z K_0(B) \too K_0(A\otimes  B)\,.
\end{equation}

\begin{example}
When $A=B=C(X)$ is a commutative unital $\Cstar$-algebra,  the multiplication map $A\otimes_\F A \to A$ is a $*$-homomorphism, which extends uniquely to the $\Cstar$-algebraic tensor product $A\otimes_\mathrm{min} A=A\otimes_\mathrm{max} A$. By \cite{ivo:unitary}*{Proposition 2.13}, the usual maximal tensor product of two unital $\Cstar$-algebras coincides with their maximal tensor product when considered as $\Cstar$-categories. Hence by applying $K_0=\Hom_{\Ho(\cM_\Mor)^{-1}}(\F,-)$ to the $*$-functor $A\otimes A\to A$ we obtain an induced homorphism
\begin{equation}\label{eq:mult}
K_0(A\otimes A)  \too K_0( A)\,.
\end{equation}
which we may then compose with \eqref{eq:pairing} to obtain a well-defined multiplication: 
\begin{equation}\label{eq:ring}
K_0(A) \otimes_\Z K_0(A) \too K_0(A) \,.
\end{equation}
\end{example}

\begin{prop}
\label{prop:ring_comm}
Let $A=C(X)$ be a unital commutative $\Cstar$-algebra.
Then the multiplication map~\eqref{eq:ring} identifies with the usual ring structure of $K_0(A)$ induced by the tensor product of vector bundles.
\end{prop}

\begin{proof}
The multiplication in terms of vector bundles is the composition of the map $K^0(X\times X) \to K^0(X)$, induced by the diagonal embedding $\Delta\colon X\to X\times X$, with the map $K^0(X)\otimes_\Z K^0(X)\to K^0(X\times X)$ induced by the external tensor product of vector bundles; here $K^0(-)$ denotes the Grothendieck group of the additive category of vector bundles. Via the Swan-Serre theorem and Theorem \ref{thm:co-representability}, the former identifies with \eqref{eq:mult} and the latter with~\eqref{eq:pairing}. The details are straightforward and left to the reader.
\end{proof}

\subsection*{Acknowledgments:} The authors are very grateful to Paul Balmer, Dan Christensen and
Amnon Neeman for the organization of the conference ``Triangulated categories and applications'' at the Banff International Research Station (Canada) where this work was initiated.


\end{document}